\documentclass{amsart}
\usepackage{graphicx}
\usepackage{amsmath}
\usepackage{amsfonts}
\usepackage{multirow}
\usepackage{color}
\usepackage{xspace}
\usepackage{array}
\usepackage{tikz}
\usetikzlibrary{positioning}
\usetikzlibrary{arrows}
\usepackage{pgfplots}
\newcommand{\rd}{\mathrm{d}}
\newcommand{\RR}{\mathbb{R}}
\newcommand{\Sb}{\mathbb{S}}
\newcommand{\zd}[2]{z^{\left( #1 \right)}_{#2}}

\newcommand{\ud}[2]{u^{\left( #1 \right)}_{ #2 }}

\newcommand{\rank}{\mathop{\mathrm{rank}}\nolimits}

\newcommand{\Range}{\mathop{\mathrm{range}}\nolimits}
\newcommand{\Null}{\mathop{\mathrm{null}}\nolimits}
\newcommand{\Car}{\mathop{\mathrm{car}}\nolimits}

\newcommand{\CM}{\mathcal{M}}
\newcommand{\dgp}{\overline{\nabla}_{\mathrm{P}}}
\newtheorem{theorem}{Theorem}[section]
\newtheorem{lemma}[theorem]{Lemma}

\newtheorem{proposition}[theorem]{Proposition}

\theoremstyle{definition}
\newtheorem{definition}[theorem]{Definition}
\newtheorem{example}[theorem]{Example}

\theoremstyle{remark}
\newtheorem{remark}[theorem]{Remark}

\makeatletter

\@addtoreset{equation}{section}
\makeatother

\title[Discrete gradient methods for differential-algebraic equations]{Linear gradient structures and discrete gradient methods for conservative/dissipative differential-algebraic equations}

\author{Shun Sato}
\address{Graduate School of Information Science and Technology, The University of Tokyo, Bunkyo-ku, Tokyo, Japan}
\email{shun\_sato@mist.i.u-tokyo.ac.jp}
\thanks{The first author was supported in part by JSPS Research Fellowship for Young Scientists.}

\date{May, 2018}


\begin{document}

\maketitle

\begin{abstract}
In this paper, we consider the use of discrete gradients for differential-algebraic equations (DAEs) 
with a conservation/dissipation law. 
As one of the most popular numerical methods for conservative/dissipative ordinary differential equations, 
the framework of discrete gradient methods has been intensively developed over recent decades. 
Although discrete gradients have been applied to several specific conservative/dissipative DAEs, 
no unified framework for DAEs has yet been constructed. 
In this paper, we move toward the establishment of such a framework, 
and introduce concepts including an appropriate linear gradient structure for DAEs. 
Then, we reveal that the simple use of discrete gradients does not imply the discrete conservation/dissipation laws. 
Fortunately, however, we can successfully construct a new discrete gradient method 
for the case of index-1 DAEs. 
We believe this first attempt provides an indispensable basis for constructing a unified framework of discrete gradient methods for DAEs. 
\end{abstract}

\section{Introduction}
\label{sec_intro}

In this paper, we consider the geometric integration of autonomous differential-algebraic equations~(DAEs) of the form
\begin{equation}\label{eq_dae}
A \dot{z} = f(z)
\end{equation}
with a conservation or dissipation law. 
The class of DAEs with a conservation or dissipation law involves mechanical systems with holonomic/nonholonomic constraints~\cite{Bloch2015} and the spatial discretization of the conservative/dissipative evolutionary equations with a mixed derivative~\cite{SM2017arx}. 
In particular, our aim is to develop a framework for constructing dissipative/conservative numerical methods using the discrete gradient~\cite{G1996}. 

For the numerical integration of conservative/dissipative ordinary differential equations (ODEs), 
numerical methods inheriting the conservation/dissipation law are known to be qualitatively better than general purpose techniques such as Runge--Kutta methods~(see, e.g., \cite{WN2016arx} for conservative ODEs and \cite{SMSF2015} for dissipative ODEs). 
The discrete gradient method is one of the most popular conservative/dissipative numerical methods 
(for other methods, see, e.g., \cite{WBN2016,WBN2017,K2016,IY2016}). 

A brief history of the discrete gradient method is summarized below. 
Though a similar technique has been reported in the literature, 
the discrete gradient was first (explicitly) defined by Gonzalez~\cite{G1996}, 
who constructed Hamiltonian-preserving numerical methods for Hamiltonian systems such as
\begin{equation}\label{eq_Hamiltonian}
	\dot{z} = J \nabla H(z), \qquad J = \begin{pmatrix} O_n & I_n \\ -I_n & O_n \end{pmatrix}, 
\end{equation}
where $ z : [0,T) \to \RR^{2n} $ is a dependent variable ($ \dot{z} $ denotes the time derivative of $z$), $I_n \in \RR^{n\times n} $ denotes the identity matrix, $O_n \in \RR^{n \times n}$ denotes the zero matrix, and $H : \RR^{2n} \to \RR $ is the Hamiltonian. 
Thanks to the skew-symmetry of $J$, conservation of the Hamiltonian $ H $ can be confirmed as follows: 
\[ \frac{\rd}{\rd t} H (z) = \langle \nabla H(z) , \dot{z} \rangle = \langle \nabla H(z) , J \nabla H(z) \rangle = 0, \]
where $ \langle \cdot, \cdot \rangle $ is the standard inner product of $ \RR^{2n} $. 
The one-step method (usually called a discrete gradient method)
\begin{equation}\label{eq_Hamiltonian_dg}
	\frac{\zd{m+1}{}-\zd{m}{}}{\Delta t} = J \overline{\nabla} H \left( \zd{m+1}{} , \zd{m}{}  \right)
\end{equation}
preserves the Hamiltonian, 
where $ \zd{m}{} \approx z(m \Delta t)  $ ($ \Delta t > 0 $ is the step size).  The discrete gradient $ \overline{\nabla} H $ is defined as follows. 

\begin{definition}\label{def_dg}
	Let $ V: \RR^d \to \RR $ be a differentiable function ($d$ is a positive integer). 
	Then, the continuous map $ \overline{\nabla} V : \RR^d \times \RR^d \to \RR^d $ is a discrete gradient of $V$ if it satisfies
	\begin{enumerate}
		\item $ \langle \overline{\nabla} V (z,z') , z - z' \rangle = V(z) - V(z') $ holds for any $ z , z' \in \RR^d $; 
		\item $ \overline{\nabla} V ( z, z ) = \nabla V(z) $ holds for any $ z \in \RR^d $.  
	\end{enumerate}
\end{definition}

Hereafter, the first property is referred to the ``discrete chain rule,'' 
as it is a discrete counterpart of the usual chain rule. 
Whereas the second property merely maintains consistency, 
the first property plays an essential role in the proof of the discrete conservation law: 
\begin{align*}
	\frac{H \big( \zd{m+1}{} \big) - H \big(\zd{m}{} \big)}{\Delta t}
	&= \left\langle \overline{\nabla} H \Big(\zd{m+1}{}, \zd{m}{} \Big) , \frac{\zd{m+1}{} - \zd{m}{}}{\Delta t} \right\rangle \\
	&= \left\langle \overline{\nabla} H \Big(\zd{m+1}{}, \zd{m}{} \Big) , J \overline{\nabla} H \Big( \zd{m+1}{} , \zd{m}{}  \Big) \right\rangle  = 0
\end{align*}
(note that the proof of the discrete conservation law is quite similar to that of its continuous counterpart). 
Note that the discrete gradient is not, in general, unique, 
and there are several method of constructing discrete gradients (see, e.g., \cite{G1996,IA1988,CGMMOOQ2012}).

Moreover, Quispel--Turner~\cite{QT1996} extended the discrete gradient method to construct conservative numerical methods for the \emph{skew-gradient system}
\begin{equation}\label{eq_lg}
	\dot{z} = S(z) \nabla V(z),
\end{equation}
where $ z: [0,T) \to \RR^d $ is a dependent variable, $V:\RR^d \to \RR$ is a differentiable function, 
and $ S : \RR^d \to \RR^{d \times d}$ is a map such that $ S(z) $ is skew-symmetric for any $ z \in \RR^d $. 
The skew-symmetry of $S(z)$ means that the conservation of $V$ can be shown in a manner similar to that of the Hamiltonian above, 
i.e., using only the ``chain rule'' and ``skew-symmetry.''
Thus, by appropriately discretizing $ S $, conservative numerical methods can be constructed in a similar manner (see Section~\ref{subsec_pre_dg}). 

Though the skew-gradient form may appear restrictive in comparison with general conservative ODEs, 
Quispel--Capel~\cite{QC1996} showed that any conservative ODEs can be written in the appropriate skew-gradient form (see Proposition~\ref{prop_MQR_conservative}). 
Therefore, in principle, the discrete gradient method can be used for any conservative ODEs. 
Moreover, McLachlan--Quispel--Robidoux~\cite{MQR1999} showed that the discrete gradient can also be used for 
ODEs with one or more conserved quantities and/or dissipated quantities (see Proposition~\ref{prop_MQR_dissipative} for the case of a single dissipated quantity, where $S(z)$ is a negative semidefinite matrix). 

Furthermore, for variational partial differential equations~(PDEs) with 
a conservation/dissipation law, 
Furihata--Mori~\cite{FM1998} (see also \cite{F1999}) devised the discrete variational derivative method (DVDM) (see~\cite{FM2011} for details). 
Though DVDM and the discrete gradient method were investigated independently, 
DVDM is now understood to be a combination of the discrete gradient method and an appropriate spatial discretization. 
Thus, evolutionary equations with a conservation/dissipation law can be regarded as the target of the discrete gradient methods 
(see \cite{CGMMOOQ2012} for details on discrete gradient methods for PDEs).

\vspace{5pt}

We now turn our attention to conservative/dissipative methods for DAEs. 

Gonzalez~\cite{G1999} showed that discrete gradients can be used for the 
conservative numerical integration of mechanical systems with a holonomic constraint. 
Until now, the discrete gradient method has been employed for various mechanical systems with holonomic and/or nonholonomic constraints having the conservation/dissipation property~(e.g., \cite{Be2006,UB2010}). 
As the  discrete gradient maintains the discrete chain rule, 
one can construct conservative/dissipative methods by some clever discretization of the holonomic/nonholonomic constraints. 

More recently, a team including the present author pointed out that 
spatial discretizations of evolutionary equations with a mixed derivative, i.e., $ u_{tx} = f(u,u_x,\dots) $ (subscripts $t$ and $x$ denote temporal and spatial partial differentiation), 
turn out to be DAEs~\cite{SM2017arx}. 
As this class of PDEs covers numerous conservative systems, 
there have been several studies on conservative numerical methods~\cite{MYM2012,FSM2016,MCFM2017,S2018+} using the spirit of DVDM. 

In view of these existing studies on DAEs, 
it is quite natural to wish for a unified framework for 
applying discrete gradient methods to DAEs. 
Surprisingly, however, there have been no reports on this topic in the literature. 
Therefore, the ultimate aim of the present author is to construct a unified framework covering discrete gradient methods for DAEs. 
In this paper, as a first attempt to establish such a framework, 
we derive several basic results. 

Firstly, we observe what happens when a DAE has a linear conserved quantity (Section~\ref{sec_linear}). 
Though this observation is not directly related to discrete gradient methods, 
it illustrates how troublesome the concept of conserved quantities can be in DAEs (see Remark~\ref{rem_dae_vfm}). 
The observation implies that, when we deal with DAEs, 
the discrete conservation turns out to pose difficulties even when the conserved quantity is linear and the DAE has index-1 (see \cite{Ascher-Petzold1998} for a definition of the differential index); 
some linear conserved quantities are automatically conserved by the implicit Euler method (and some implicit Runge--Kutta methods), whereas others are not. 
This should be somewhat surprising, because it is widely known that, in ODEs, all linear conserved quantities are automatically conserved by all explicit and implicit Runge--Kutta methods. 
This difficulty in DAEs is caused by a lack of simple criteria for the conserved quantity: 
recall that the conserved quantity in ODEs can be characterized by the orthogonality condition $ \langle \nabla V(z) , f(z) \rangle = 0 $. 

\begin{remark}\label{rem_dae_vfm}
	As shown in Section~\ref{subsec_pre_dae}, a DAE is actually an ``implicit definition'' of a vector field $v$ on some manifold $ \CM $ (see Definition~\ref{def_regular}). 
	In this sense, one may feel that the rich results reported in previous studies (see Olver~\cite{Olver1995}, for example) are sufficient to deal with conservation laws in DAEs. 
	Unfortunately, however, we usually conduct the numerical integration of DAEs without explicitly detecting either the vector field $v$ or manifold $\CM$ (backward difference formulae (BDF) methods are typical examples; see~\cite{Ascher-Petzold1998}). 
	In particular, when we deal with cases in which the dimension of $ \CM $ is nearly $d$ (such as the spatial discretizations of PDEs with a mixed derivative), the computation based on $v$ and $ \CM$ is not practical because of the lack of sparsity. 
	In this sense, the criterion that ``$V$ is the conserved quantity if and only if $ \langle \nabla V(z), v(z) \rangle = 0  $ hold for any $ z \in \CM$'' is useless for our aim. 
	
	Moreover, as we do not want to step into $ v $ and $\CM $, 
	several elegant results on the geometric numerical integration of vector fields on manifolds (e.g., Celledoni--Owren~\cite{CO2014} for the conservative case and Celledoni--Eidnes--Owren--Ringholm~\cite{CEOR2018arx} for the dissipative case) cannot be employed. 
\end{remark}

To overcome this difficulty, 
we propose the concept of ``proper functions'' for DAEs (Section~\ref{sec_proper}). 
Though the definition of ``properness'' is unusual at first sight, 
we believe that it is indispensable for exploring the conservation/dissipation law of DAEs 
for the following reasons (each point will be described in Section~\ref{subsec_def_proper}): 
\begin{itemize}
	\item[(a)] it is a natural extension of the linear case.
 	\item[(b)] it is a natural extension of the ODE case.
	\item[(c)] it has simple criteria for conservation/dissipation laws.
	\item[(d)] it forms a sufficiently large subclass of functions.
\end{itemize}

Actually, the lack of the simple criteria for conservation/dissipation laws has already been identified in studies of Lyapunov functions for DAEs~\cite{B1987,LT2012}. 
The Lyapunov function is a dissipated quantity with several other properties (which depend on the context and expected consequence; see, e.g., \cite{Robinson,KR2011,V2002}). 
In the literature, subclasses of Lyapunov functions whose time differentiation can be expressed without information on $v$ and $\CM$, have been considered. 
Though these attempts were successful to a certain extent, they did not consider how restrictive the associated assumptions were. 
The set of proper functions is a new candidate for such a subclass (see point (b) in the list above), 
and its distinct advantage is one aspect of point (d): the assumption of properness does not lose generality (see Proposition~\ref{prop_proper}). 
This advantage strongly relies on the autonomous nature of the target DAE~\eqref{eq_dae} 
(see Remark~\ref{rem_lyapunov} for details).

Based on the concept of proper functions, 
we show that the conservation (resp. dissipation) law of a DAE extension~\eqref{eq_lgdae} of linear gradient form~\eqref{eq_lg} can be characterized by the skew-symmetry (resp. negative semidefiniteness) of a matrix (see Section~\ref{sec_lg}). 
This result is a natural DAE extension of McLachlan--Quispel--Robidoux~\cite{MQR1999}. 
Thus, it implies that the linear gradient DAE~\eqref{eq_lgdae} is an appropriate extension of the linear gradient form~\eqref{eq_lg}, 
and indicates the possibility of a unified framework for discrete gradient methods for DAEs. 

We then consider the use of discrete gradients for linear gradient DAEs~\eqref{eq_lgdae} in Section~\ref{sec_dg} 
(as the situation is the same when we deal with the dissipation law, 
we focus on the conservative case in this part ). 
However, because the numerical solution can break the constraint of the DAEs, 
the assumption of the discrete conservation law becomes somewhat restrictive. 
Still, existing conservative numerical methods for DAEs are covered by this case, 
and discrete conservation laws are maintained thanks to their special structures 
(see Examples~\ref{ex_hs_dg} and \ref{ex_ad_dg}). 

Moreover, for index-1 cases, 
we propose a new discrete gradient method 
that conserves both the desired conserved quantity and the constraint (Section~\ref{sec_dg_index1}). 
We introduce a new discrete gradient that is compatible with proper functions, and 
employ a reformulation~\eqref{eq_dae_red} that is appropriate for the simultaneous conservation of constraints and the conserved quantity. 

Discrete gradient schemes are numerically examined in Section~\ref{sec_num} 
using the sinh-Gordon equation~\eqref{eq_shG} as an example. 
The proposed discrete gradient method successfully conserves the conserved quantity and the constraint. 
However, as naturally expected, the computational cost becomes expensive. 

Though a more sophisticated framework for the discrete gradient method for DAEs is left for future work, 
the author believes that the contribution described in this paper will play an important role in these studies. 

The remainder of this paper is organized as follows. 
Section~\ref{sec_pre} is devoted to preliminaries such as discrete gradient methods for ODEs and some basic concepts of DAEs. 
Sections~\ref{sec_linear}--\ref{sec_num} have been essentially described above. 
The paper concludes in Section~\ref{sec_cr}.

\section{Preliminaries}
\label{sec_pre}

\subsection{Linear gradient systems and conservation/dissipation laws}
\label{subsec_pre_lg}

First, we define the conserved and dissipated quantities for ODEs of the form
\begin{equation}\label{eq_ode}
\dot{z} = f(z).
\end{equation}

\begin{definition}[Conserved quantity~\protect{(e.g., \cite[Definition~7.35]{Olver1995})}]\label{def_fi_ode}
	Let $ V : \RR^d \to \RR $ be a $ C^r $ function with $ r \ge 1 , \ d > 1$. 
	Then, $ V $ is called a conserved quantity of~\eqref{eq_ode} if $ \frac{\rd}{\rd t} V(z(t)) = 0 $ holds for any solution $z$ of~\eqref{eq_ode}.    
\end{definition}

\begin{definition}[Dissipated quantity~\protect{\protect{(cf. \cite[Definition~10.11]{Robinson}})}]\label{def_lf_ode}
	Let $ V : \RR^d \to \RR $ be a $ C^r $ function with $ r \ge 1 , \ d \ge 1$. 
	Then, $ V $ is called a dissipated quantity of~\eqref{eq_ode} if $ \frac{\rd}{\rd t} V(z(t)) \le 0 $ holds for any solution $z$ of~\eqref{eq_ode}.  
\end{definition}

When the ODE~\eqref{eq_ode} is written in the linear gradient form~\eqref{eq_lg}, 
i.e., $ f(z) = S(z) \nabla V(z) $, 
$ V $ is a conserved (resp. dissipated) quantity if $ S : \RR^d \to \RR^{d\times d} $ satisfies ``$ S(z) $ is skew-symmetric (resp. negative semidefinite) for any $ z \in \RR^d $.'' 
The matrix $ S $ is said to be skew-symmetric if $ S^{\top} = - S $ holds 
($ S^{\top} $ denotes the transpose of $S$), 
and is said to be negative semidefinite if $ \langle z, S z \rangle \le 0 $ holds for any $ z \in \RR^d $. 

McLachlan--Quispel--Robidoux~\cite{MQR1999} showed the converse:  
if the ODE~\eqref{eq_ode} has a conserved (resp. dissipated) quantity, 
it can be rewritten in the appropriate linear gradient form. 

\begin{proposition}[\protect{\cite[Proposition~2.1]{MQR1999}}]\label{prop_MQR_conservative}
	Let $ f : \RR^d \to \RR^d $ be a $C^r$ map with $ r \ge 1 , \  d>1 $, and $ V : \RR^d \to \RR $ be a $ C^{r+1} $ function such that $ \langle f(z) , \nabla V(z) \rangle = 0 $ for any $ z \in \RR^d $. 
	Then there exists a skew-symmetric matrix function $ S $ such that $C^r$  and $ f = S  \nabla V $ on the domain $ \{  z \in \RR^d \mid \nabla V(z) \neq 0  \} $. 
	Moreover, $S$ can be chosen so as to be bounded near every non-degenerate critical point. 
	Then,  $S$ is locally bounded if $V$ is a Morse function, that is, a smooth function in which all critical points are non-degenerate. 
\end{proposition}

\begin{proposition}[\protect{\cite[Proposition~2.8]{MQR1999}}]\label{prop_MQR_dissipative}
	Let $ f : \RR^d \to \RR^d $ be a $C^r$ map with $ r \ge 1 , \  d\ge1 $, and $ V : \RR^d \to \RR $ be a $ C^{r+1} $ function such that $ \langle f(z) , \nabla V(z) \rangle \le 0 $ for any $ z \in \RR^d $. 
	Then there exists a symmetric negative definite matrix function $ S $ such that $C^r$ and $ f = S  \nabla V $ on the domain $ \{  z \in \RR^d \mid \nabla V(z) \neq 0  \} $. 
\end{proposition}

The above propositions show that 
the linear gradient form with the skew-symmetric (resp. negative semidefinite) matrix function $ S $ is a sufficiently large class for considering conservative (resp. dissipative) systems. 
Therefore, it is meaningful to consider the conservative/dissipative temporal discretization of such systems.  
This is realized using discrete gradients in the next section.

\subsection{Discrete gradient methods}
\label{subsec_pre_dg}

For the linear gradient ODEs~\eqref{eq_lg}, 
the discrete gradient method is defined as 
\begin{equation}\label{eq_lg_dg}
\frac{\zd{m+1}{} - \zd{m}{}}{\Delta t } = \overline{S} \left( \zd{m+1}{}, \zd{m}{} \right) \overline{\nabla} V  \left( \zd{m+1}{}, \zd{m}{} \right), 
\end{equation}
where $ \overline{S} : \RR^d \times \RR^d \to \RR^{ d \times d } $ is a consistent approximation of $ S $ (i.e., $ \overline{S} (z,z) =  S(z) $) 
such that $ \overline{S} (z,z') $ is skew-symmetric (resp. negative semidefinite) for any $ z ,z' \in \RR^d $. 
As
\begin{align*}
\frac{V \big(\zd{m+1}{} \big) - V \big( \zd{m}{} \big) }{\Delta t }
&= \left\langle \overline{\nabla} V \Big(\zd{m+1}{}, \zd{m}{} \Big) , \frac{\zd{m+1}{} - \zd{m}{}}{\Delta t} \right\rangle \\
&= \left\langle \overline{\nabla} V \Big(\zd{m+1}{}, \zd{m}{} \Big) , \overline{S}  \Big( \zd{m+1}{} , \zd{m}{}  \Big)  \overline{\nabla} V \Big( \zd{m+1}{} , \zd{m}{}  \Big) \right\rangle 
\end{align*}
holds, 
the discrete gradient method~\eqref{eq_lg_dg} is conservative (resp. dissipative) 
because of the skew-symmetry (resp. negative semidefiniteness) of $ \overline{S}$. 

Note that the discrete gradient (see Definition~\ref{def_dg}) is not unique, 
and several methods of constructing discrete gradients are known. 
For example, the average vector field~\cite{QM2008}
\begin{equation}\label{eq_avf}
\overline{\nabla}_{\mathrm{AVF}} V (z,z') = \int_0^1 \nabla V \left( (1 - \xi) z + \xi z' \right) \rd \xi
\end{equation}
is known to be a discrete gradient.

\subsection{Moore--Penrose inverse matrix}
\label{subsec_pre_mp}

In this section, some basic properties of the Moore--Penrose inverse are summarized. 
As the matrix $A$ that appears in DAE~\eqref{eq_dae} is assumed to be singular, 
we frequently consider its Moore--Penrose inverse in the present paper. 

For a matrix $ A \in \RR^{d \times d} $, the Moore--Penrose inverse $ A^{\dagger} $ of $ A $ is defined as 
the unique matrix satisfying
\begin{align*}
A A^{\dagger} A &= A,&
A^{\dagger} A A^{\dagger} &= A^{\dagger}, &
(A^{\dagger} A)^{\top} &= A^{\dagger} A, &
(A A^{\dagger})^{\top} &= A A^{\dagger}
\end{align*} 
(see \cite{IsraelGreville2003} for details on generalized inverses). 
Note that, though the Moore--Penrose inverse is also defined for non-square matrices, 
we only use the square case in this paper. 

Let $ \Null (A) $ and $ \Range (A) $ denote the null space and the range of $A$. 
Then, the career $ \Car (A) $ is defined as $ \Car (A) = \Null (A)^{\perp} $, 
where $ X^{\perp} $ denotes the orthogonal complement of the linear subspace $ X$. 

For the Moore--Penrose inverse $ A^{\dagger} $, 
the relations $ \Range (A^{\dagger}) = \Car (A) $ and $ \Car (A^{\dagger} ) = \Range (A) $ hold. 
Moreover, $ A^{\dagger} A $ and $ A A^{\dagger} $ are orthogonal projectors on $ \Range (A^{\dagger}) $ and $ \Range (A) $. 
Thus, the relation $ \{ x \mid A x = b \} = \{ A^{\dagger} b + e \mid e \in \Null (A) \} $
holds for any $ b \in \Range (A) $.

\subsection{Basic concepts of DAEs}
\label{subsec_pre_dae}

In this section, several basic concepts of DAEs are reviewed (see, e.g., \cite{Ascher-Petzold1998,Burger-Gerdts2017} for details on DAEs). 
The most general form of autonomous DAEs is $ \psi (z,\dot{z}) =0 $ for some function $\psi$, 
but this can be rewritten in the form~\eqref{eq_dae} by introducing a new variable $ v = \dot{z} $. 
Therefore, we focus on DAEs of the form~\eqref{eq_dae} in this paper. 

The concept of regular DAEs was introduced by Reich~\cite{R1990} as follows.

\begin{definition}{\protect{\cite[Definition~4]{R1990}}}\label{def_regular}
	A DAE~\eqref{eq_dae} is said to be regular 
	if there is a differential submanifold $\CM $ of $ \RR^d $ and a vector field $ v: \CM \to T \CM $ 
	such that a differentiable mapping $ w : [0,T) \to \CM $ is a solution of the vector field $v$ if and only if the mapping $ z : = j \circ w : [0,T) \to \RR^d $ is a solution of the DAE~\eqref{eq_dae}. 
	Here, $ j : \CM \to \RR^d $ is the natural projection. 	
	Then, the manifold $\CM$ is called the configuration space, 
	and the vector field $ v $ is the corresponding vector field of the DAE~\eqref{eq_dae}. 
\end{definition}

Throughout this paper, the DAE is assumed to be regular. 
Furthermore, for simplicity, we assume $ \CM $ can be expressed as $ \CM= \{ z \in \RR^d \mid g_i (z) = 0 \ (i = 1, \dots,k) \} $ by some functions $ g_i \ ( i = 1, \dots, k) $. 

It is important to note that, in numerical computation, 
we often do not use the corresponding vector field $v$ and the configuration space $ \CM$ (Remark~\ref{rem_dae_vfm}). 
For example, the implicit Euler method (the simplest of the usual BDF methods; see \cite{Ascher-Petzold1998} for details) for the DAE~\eqref{eq_dae} can be written as 
\begin{equation}\label{eq_ie}
 A \frac{ \zd{m+1}{} - \zd{m}{} }{\Delta t } = f \big( \zd{m+1}{} \big). 
 \end{equation}
Under a certain assumption, this difference equation provides us with an approximation of the exact solution for the DAE~\eqref{eq_dae}, 
but clearly we do not use either the corresponding vector field~$v$ or the configuration manifold~$\CM$.

Although the well-known DAE form 
\[ \begin{cases}
\dot{y} = \psi(y,z) \\
0 = \phi(y,z)
\end{cases} \]
has the explicit ``algebraic'' constraint $ \phi (y,z) = 0 $, 
the target DAE~\eqref{eq_dae} does not have such an explicit constraint. 
However, the target DAE~\eqref{eq_dae} has this constraint implicitly. 
To see this, we introduce the orthonormal basis $ \{ b_i \}_{i=1}^{\ell} $ of $ (\Range (A) )^{\perp} $ ($ \ell := d - \rank A $), 
and note that $ 0 = \langle b_i , A \dot{z} \rangle = \langle b_i , f(z) \rangle $
holds for any $ i = 1, \dots, \ell $. 
In other words, the DAE~\eqref{eq_dae} has the implicit constraint 
\begin{equation}\label{eq_dae_ic}
B^{\top} f(z) = 0,
\end{equation}
where $ B = ( b_1, \dots, b_{\ell}) \in \RR^{d \times \ell} $. 
Therefore, in general, 
$ \CM \subseteq \CM_B := \{ z \in \RR^d \mid B^{\top} f (z) = 0 \} $ holds. 

When we deal with a higher (differential) index DAE, there are \emph{hidden constraints}. 
Though we do not define the index here (see \cite{Ascher-Petzold1998} for details), 
it should be noted that DAEs have a uniform index-1 if and only if $ \CM = \CM_B$ holds. 


We describe this point using the Moore--Penrose inverse $A^{\dagger} $ of $A$. 
For any solution $ z:[0,T) \to \RR^d $ of the DAE~\eqref{eq_dae}, 
there exists some $ e :[0,T) \to \Null (A) $ such that 
\begin{equation}\label{eq_vf_ker}
\dot{z} (t) = A^{\dagger} f(z(t)) + e (t), 
\end{equation}
because of the property of the Moore--Penrose inverse. 
The configuration space $ \CM $ is actually the maximal manifold for which
$ T_z \CM \cap \{ A^{\dagger} f(z) + e \mid e \in \Null (A)  \} $ is a singleton for any $ z \in \CM $.

\section{Note on linear conserved quantities in DAEs}
\label{sec_linear}

In this section, we consider the discrete conservation of linear conserved quantities in DAEs. 
Though this discussion is not directly related to discrete gradient methods, 
the observations in this section illustrate the complexity of the conservation laws in DAEs. 

For the case of ODEs, 
it is widely known that the discrete preservation of linear conserved quantities is relatively simple
 (see \cite[Section IV.1]{HLW2010}). 
For example, all explicit and implicit Runge--Kutta methods automatically conserve linear conserved quantities~\cite[IV Theorem~1.5]{HLW2010}. 
This significant property holds because the existence of a linear conserved quantity $ V(z) = \langle \gamma , z \rangle $ ($ \gamma \in \RR^d $ is a constant vector) implies that the vector field $f(z)$ is always orthogonal to the constant vector~$ \gamma $. 

Thus, in the case of DAEs, one may feel that the linear conserved quantities should be easily replicated in numerical methods. 
Unfortunately, however, this is not the case. 
To clarify this point, we first define the concept of conserved quantities in DAEs as follows. 
Note that the definition below is quite natural in view of Definition~\ref{def_fi_ode}, 
and is consistent for the concept of ``invariants'' for the vector fields on the manifold (see \cite[Definition~2.29]{Olver1995}).

\begin{definition}[Conserved quantity]\label{def_fi}
	Let $ V : \RR^d \to \RR $ be a $C^r$ map with $ r \ge 1$, $ d > 1 $. 
	Then, $V$ is said to be a \emph{conserved quantity} if $ \frac{\rd}{\rd t} V(z(t)) = 0 $ holds for any solution $z $ of DAE~\eqref{eq_dae}. 
\end{definition}

As all solutions $ z $ of the DAE~\eqref{eq_dae} satisfy $ z(t) \in \CM $ (while $ \RR^d $ is filled with solutions of the ODE~\eqref{eq_ode}), 
the existence of a linear conserved quantity $ V(z) = \langle \gamma , z \rangle $ only implies that 
the vector field $v(z) $ (recall Definition~\ref{def_regular}) is orthogonal to $ \gamma $ for any $ z \in \CM $. 
The following two changes from the ODE case should be emphasized: 
\begin{enumerate}
\item The orthogonality is only given in terms of the corresponding vector field $v$ (see \cite[Theorem~2.74]{Olver1995}), which is not explicitly given in DAE cases (recall that $v$ is not generally used in numerical methods for DAEs). 
\item The orthogonality is only satisfied on the configuration manifold $\CM$. 
\end{enumerate}

The first point makes it difficult to derive some necessary and sufficient condition for the linear conservation law by means of $f$ (and $A$). 
Moreover, when the numerical solution $ \zd{m}{} $ is outside of the configuration manifold $\CM$, 
the second point becomes troublesome. 
The latter is not actually a major issue in the index-1 case (see Section~\ref{subsec_pre_dae}), 
but the former is more delicate. 

Still, using~\eqref{eq_vf_ker}, we see that
\begin{equation}\label{eq_chain_linear}
 \frac{\rd}{\rd t} V(z(t)) = \langle \gamma, \dot{z} (t) \rangle = \langle \gamma , A^{\dagger} f(z(t)) + e(t) \rangle = \langle \gamma , A^{\dagger} f(z) \rangle + \langle \gamma , e(t) \rangle. 
 \end{equation}
Therefore, if the linear conserved quantity $V$ satisfies $ \gamma \in \Car (A) = (\Null (A) )^{\perp} $, 
we can overcome the former difficulty; 
the conservation of $V$ implies $ \langle \gamma , A^{\dagger} f(z) \rangle = 0 $ for any $ z \in \CM $. 
Using the fact that the implicit Euler method~\eqref{eq_ie} satisfies $ \zd{m}{} \in \CM $ for uniform index-1 DAEs, we obtain the following proposition. 

\begin{proposition}\label{prop_lc_dae}
	Let $ \zd{0}{} \in \CM $ be the initial condition and $ \zd{m}{} $ be the solution of the implicit Euler method~\eqref{eq_ie} ($ m = 1,2, \dots $). 
	Suppose that $ V(z) = \langle \gamma , z \rangle $ is a conserved quantity of the DAE~\eqref{eq_dae} satisfying $ \gamma \in \Car (A) $, and that the DAE~\eqref{eq_dae} has a uniform index-1. 
	Then, for any $ m = 0, 1, \dots $, $ V(\zd{m+1}{} ) = V(\zd{m}{}) $ holds. 
\end{proposition}

\begin{proof}
	Similar to~\eqref{eq_vf_ker}, for any $ m = 0,1, \dots $, 
	there exists some $ e^{(m+1)} \in \Null(A) $ such that 
	\[ \frac{\zd{m+1}{} - \zd{m}{} }{\Delta t } = A^{\dagger} f(\zd{m+1}{} ) + e^{(m+1)} \]
	holds. 
	Therefore, we see that 
	\begin{align*}
	\frac{V(\zd{m+1}{}) - V(\zd{m}{})}{\Delta t} 
	&= \left\langle \gamma , \frac{\zd{m+1}{} - \zd{m}{}}{\Delta t} \right\rangle
	= \left\langle \gamma , A^{\dagger} f( \zd{m+1}{} ) \right\rangle + \left\langle \gamma, e^{(m+1)} \right\rangle,
	\end{align*}
	in which the right-hand side vanishes because $ \zd{m+1}{} \in \CM $, $ e^{(m+1)} \in \Null (A)$, and $ \gamma \in \Car (A) $. 
\end{proof}

A similar proposition holds for a subclass of implicit Runge--Kutta methods (the proof is similar). 
These discrete conservation laws are just examples, 
but the fact implies that the linear conserved quantity is not difficult to replicate in numerical methods when $ \gamma \in \Car (A)  $. 

Unfortunately, however, this approach truly relies on the condition $ \gamma \in \Car (A) $. 
As shown in~\eqref{eq_chain_linear}, even for the continuous case, 
the proof of the conservation of such a linear conserved quantity depends on $ e(t) $, which is equivalent to the vector field $v$ (when we obtain $e(t)$, the vector field $v$ can be constructed by $ A^{\dagger} f(z) + e $).
Actually, the nonlinear DAE (obtained by a very coarse spatial discretization of the modified Hunter--Saxton equation; see~\cite{S2018+})
\begin{equation}\label{eq_smHS}
\begin{pmatrix} -1 & 1 & 0 \\ 0 & -1 & 1 \\ 1 & 0 & -1  \end{pmatrix} \dot{z} = \frac{1}{2} \begin{pmatrix} 1 & 1 & 0 \\ 0 & 1 & 1 \\ 1 & 0 & 1  \end{pmatrix} \begin{pmatrix} z_1 ( 1 + 2 z_1 - z_2 - z_3 ) \\ z_2 (1 + 2 z_2 - z_1 - z_3) \\ z_3 ( 1 + 2 z_3 - z_1 - z_2)  \end{pmatrix} - \frac{1}{2} \begin{pmatrix} (z_2 - z_1)^2 \\ (z_3- z_2)^2 \\ (z_1 - z_3)^2  \end{pmatrix}
\end{equation}
illustrates the difficulty of the numerical preservation of such a linear conserved quantity: 
\begin{itemize}
	\item it has a linear conserved quantity $ V (z) =  \langle \mathbf{1} , z \rangle $ with $ \mathbf{1} := (1,1,1)^{\top}\notin \Car (A) $; 
	\item it has a uniform index-1; but 
	\item the implicit Euler method~\eqref{eq_ie} fails to conserve $V$. 
\end{itemize}
Though we omit the numerical simulation here, 
this fact illustrates the difficulty involved with the conservative numerical integration of DAEs
(the reason for this phenomenon is described in Example~\ref{ex_smHS}).

\section{Proper functions}
\label{sec_proper}

As shown in the previous section, 
conserved quantities of DAEs can be separated into two classes, 
i.e., relatively easy ones ($ \gamma \in \Car(A)$ in the linear case) and difficult ones. 

Therefore, before we deal with conservation/dissipation laws in DAEs, 
we extend the former class to the general nonlinear case. 
To this end, we introduce the concept of ``properness'' for DAEs in Section~\ref{subsec_def_proper}, 
and present some examples in Section~\ref{subsec_example_proper}. 

\subsection{Concept of proper functions and its advantages}
\label{subsec_def_proper}

\begin{definition}[Proper functions]
	Let $ V : \RR^d \to \RR $ be a differentiable function. 
	Then, $ V $ is said to be \emph{proper} for the DAE~\eqref{eq_dae}
	if $ \nabla V(z) \in \Car (A) $ holds for any $ z \in \CM $. 	  
\end{definition}

Though the above definition may appear unnatural, 
the concept of proper functions has the following desirable properties: 
\begin{itemize}
	\item[(a)] it is a natural extension of the linear case: \\
	in the linear case ($V$ is a linear function), $V$ is proper if and only if $ \gamma \in \Car(A) $, i.e., naturally includes the case in Proposition~\ref{prop_lc_dae}; 
	\item[(b)] it is a natural extension of the ODE case: \\
	ODEs correspond to the case where $A$ is nonsingular, and so all functions are proper; 
	\item[(c)] it has simple criteria for conservation/dissipation:\\
	using~\eqref{eq_vf_ker}, for some proper function $V$ and a solution $z$ of the DAE~\eqref{eq_dae}, we see that
	\begin{equation}\label{eq_chain_proper}
	\frac{\rd}{\rd t} V(z(t)) = \left\langle \nabla V(z) , \dot{z} (t) \right\rangle = \left\langle \nabla V(z) , A^{\dagger} f(z) + e(t) \right\rangle =  \left\langle \nabla V(z) , A^{\dagger} f(z) \right\rangle,
	\end{equation}
	which implies that the conservation/dissipation of $V$ can be characterized by the value of the right-hand side (note that the right-hand side does not depend on $v$);
	\item[(d)] it forms a sufficiently large subclass of functions: 
	\begin{itemize}
		\item a lot of known conserved quantities turn out to be proper, as shown in Section~\ref{subsec_example_proper};
		\item every function can be transformed to a proper function without changing its values on $\CM$ (Proposition~\ref{prop_proper}). 
	\end{itemize}
\end{itemize}

The final point is summarized in the following proposition. 
As all solutions of the DAE~\eqref{eq_dae} belong to the configuration manifold $\CM$, 
the value of the function $V$ outside $\CM$ does not have any impact on the conservation/dissipation laws. 
Therefore, the following proposition shows that the class of proper functions forms 
a sufficiently large subclass. 

\begin{proposition}\label{prop_proper}
	Let $ \tilde{V} : \RR^d \to \RR $ be a $ C^{r+1} $ function with $ r \ge 1 $. 
	Suppose that $ g_i : \RR^d \to \RR $ is a $ C^{r+1} $ function for each $ i = 1, \dots, p $. 
	Then, there exists a $ C^{r} $ proper function $ V: \RR^d \to \RR $ such that $ V(z) = \tilde{V} (z) $ holds for any $ z \in \CM $. 
\end{proposition}

\begin{proof}
	Note that, for any functions $ c_i :\RR^d \to \RR \ ( i = 1 ,\dots , k )$, the function $V$ defined by 
	$ V (z) := \tilde{V} (z) + \sum_{i=1}^k c_i (z) g_i (z) $ satisfies $ V = \tilde{V} $ on $ \CM $. 
	Moreover, when each $ c_i $ is differentiable, for any $ z \in \CM $, the gradient of $V $ can be expressed as
	\[ \nabla V (z) = \nabla \tilde{V} (z) + \sum_{i=1}^k ( c_i (z) \nabla g_i (z) + g_i (z) \nabla c_i (z) ) = \nabla \tilde{V} (z) + \sum_{i=1}^k  c_i (z) \nabla g_i (z). \]
	Hence, it is sufficient to show that there exists an appropriate definition of $ c(z) = (c_1 (z) ,\dots , c_k (z))^{\top} $ such that $ \nabla V (z) \in \Car (A) $ holds for any $ z \in \CM $. 
	In other words, we should find $ c (z) $ satisfying the linear equation 
	\begin{equation}\label{eq_coeff}
	W(z) c (z) = - \begin{pmatrix} e_1 & \cdots & e_{\ell} \end{pmatrix}^{\top} \nabla \tilde{V} (z), 
	\end{equation}
	where $ W(z) = \begin{pmatrix} e_1 & \cdots & e_{\ell} \end{pmatrix}^{\top} \begin{pmatrix} \nabla g_1 (z) & \cdots & \nabla g_k (z) \end{pmatrix} $, and $\{ e_i \}_{i=1}^{\ell} $ is an orthonormal basis of $ \Null (A) $. 
	
	To ensure the existence of an appropriate $ C^r $ map $ c $, 
	we show that $ W (z) $ has full row rank. 
	Note that the tangent space $ T_z \CM $ of $ \CM $ at $ z $ can be expressed as
	$ T_z \CM = \{ w \in \RR^d \mid \langle \nabla g_i (z), w \rangle = 0 \ ( i = 1, \dots , k) \} $. 
	As the DAE~\eqref{eq_dae} is assumed to be regular, 
	$ T_z \CM \cap \{ w \in \RR^d \mid A w = f(z) \} $ is a singleton for any $ z \in \CM $. 
	Summing up, for each $ z \in \CM$, the linear equation
	\[  \begin{pmatrix}
	A^{\top} & 
	\nabla g_1 (z) &
	\cdots & 
	\nabla g_k (z)
	\end{pmatrix}^{\top} v = \begin{pmatrix} f(z)^{\top} & 0 & \cdots & 0 \end{pmatrix}^{\top}  \]
	has exactly one solution $ v $. 
	Therefore, the matrix $ \begin{pmatrix} A^{\top} & \nabla g_1 (z) & \cdots & \nabla g_k (z) \end{pmatrix} $ has full row rank. 
	In addition, we see that
	\[ \begin{pmatrix} e_1 & \cdots & e_{\ell} \end{pmatrix}^{\top} \begin{pmatrix}
	A^{\top} & 
	\nabla g_1 (z) &
	\cdots & 
	\nabla g_k (z)
	\end{pmatrix} = \begin{pmatrix} O_{\ell,d} & - L (z) \end{pmatrix}, \]
	where $ O_{\ell, d} \in \RR^{\ell \times d} $ is the zero matrix. 
	As the left-hand side has full row rank, $ \rank W (z) = \ell $ holds. 
	
	Because $ W (z) $ has full row rank, 
	\[ c(z) = - ( W(z) )^{\top} \left( W (z) ( W(z))^{\top} \right)^{-1} \begin{pmatrix} e_1 & \cdots & e_{\ell} \end{pmatrix}^{\top} \nabla \tilde{V} (z) \]
	is a solution of~\eqref{eq_coeff}, which shows the proposition. 
\end{proof}

\begin{remark}
	Note that, for each function $ \tilde{V} $, 
	the proper function $ V : \RR^d \to \RR $ satisfying $ V(z) = \tilde{V} (z) \ (z \in \CM) $ is not unique. 
	For instance, as we will see in Example~\ref{ex_hs_con}, 
	the Hamiltonian system with holonomic constraints includes a proper conserved quantity, 
	and the so-called augmented Hamiltonian is also a proper conserved quantity whose value coincides with that of the original Hamiltonian on $ \CM$. 
	
	In this sense, the proper function is not a ``normal form'' of the class of functions having the same value on $ \CM $. 
	However, using this diversity in an appropriate manner, 
	one may achieve a simple linear gradient structure, resulting in better numerical methods. 
	In fact, as shown in Example~\ref{ex_hs_lg}, 
	the linear gradient structure can dramatically depend on the choice of proper functions. 
\end{remark}

To partly confirm the meaning of properness, 
we consider index-1 DAEs. 
In this case, as shown in the lemma below, 
none of the constraints is proper, 
and any nontrivial combination of a proper function and constraints is not proper. 

\begin{lemma}\label{lem_prop_constraint}
	Suppose that the DAE~\eqref{eq_dae} has a uniform index-1. 
	Then, $ \nabla g_i (z) \notin \Car (A) $ holds for each $ i = 1 ,\dots ,\ell $ and any $ z \in \CM $, 
	i.e., the constraints $g_i $ are not proper. 
\end{lemma}

\begin{proof}
	As stated in the proof of Proposition~\ref{prop_proper}, 
	since the DAE~\eqref{eq_dae} is assumed to be regular, 
	the matrix $ (A^{\top} , \nabla g_1 (z) , \dots , \nabla g_{\ell} (z) ) $ has full row rank for any $ z \in \CM $. 
	This fact implies the lemma, because the span of the column vectors of $ A^{\top} $ coincides with $ \Car (A) $. 
\end{proof}

\begin{remark}\label{rem_hi_proper}
	When we deal with a higher index DAE, 
	there are several hidden constraints, i.e., $ \CM $ is a proper subset of $ \CM_B $. 
	In this case, some constraints can be proper. 
	This means that, in considering the conservation/dissipation laws, 
	the set of proper functions is an unnecessarily large subset of functions 
	(these constraints are regarded as proper conserved/dissipated quantities). 
	However, as a mathematical definition of the ``desired conserved/dissipated quantities'' seems to be difficult, and 
	one can typically distinguish the appropriate conserved/dissipated quantities from constraints in physical problems, 
	we believe the concept of properness is a good candidate. 
	In this sense, though the later discussion on conservation/dissipation law sometimes (unintentionally) includes the constraints, we are not especially concerned by this. 
\end{remark}

\subsection{Examples of proper functions}
\label{subsec_example_proper}

Through the following examples, 
we show that the known conserved quantity is often proper. 
Prior to physical examples, we use the artificial example in~\eqref{eq_smHS}. 

\begin{example}\label{ex_smHS}
	The nonlinear DAE~\eqref{eq_smHS} has three conserved quantities
	\begin{align*}
	H(z) &= \frac{1}{2} \left( (z_2 - z_1)^2 + (z_3 - z_2)^2 + (z_1 - z_3)^2 \right), \\
	V(z) &= z_1 + z_2 + z_3, \\
	g(z) &= z_1 + z_2 + z_3 + \frac{1}{2} \left( (z_2 - z_1)^2 + (z_3 - z_2)^2 + (z_1 - z_3)^2 \right). 
	\end{align*}
	Because $ g(z) = H(z) + V(z) $, these quantities are functionally dependent. 
	As the configuration manifold $ \CM = \{ z \in \RR^3 \mid g(z) = 0 \} $, 
	$g$ turns out to be a constraint of the DAE~\eqref{eq_smHS}. 
	Therefore, Lemma~\ref{lem_prop_constraint} implies that $g$ is not proper. 
	Moreover, as mentioned in Section~\ref{sec_linear}, $V$ is not proper. 
	
	However, we see that 
	\[ \nabla H(z) = \begin{pmatrix} 2 & -1 & -1 \\ -1 & 2 & -1 \\ -1 & -1 & 2 \end{pmatrix} \begin{pmatrix} z_1 \\ z_2 \\ z_3 \end{pmatrix} \in \Car \left( A \right), \]
	which indicates that $H$ is proper. 
	
	In view of this classification, the fact that ``the implicit Euler method fails to conserve $V$'' can be explained as follows. 
	As $ H = - V $ holds on $\CM$, the linear conserved quantity $ V $ is ``equivalent'' to the quadratic conserved quantity $H$. 
	In addition, the implicit Euler method generally does not replicate the quadratic conserved quantity. 
	Therefore, the implicit Euler method fails to conserve $V$, which is just a shadow of the proper quadratic conserved quantity~$ H $. 
\end{example}

\begin{example}\label{ex_hs_con}
	In this example, the Hamiltonian system with holonomic constraints
	\begin{equation}\label{eq_hs}
	\begin{cases}
	\dot{q} = \frac{\partial H}{\partial p},\\
	\dot{p} = -\frac{\partial H}{\partial q} - (J g (q) )^{\top} \lambda,\\
	0 = g (q),
	\end{cases}
	\hspace{-7pt}
	\iff 
	\begin{pmatrix}
	I_n &  &  \\ 
	& I_n &  \\ 
	&  & O_h
	\end{pmatrix} 
	\begin{pmatrix} \dot{q} \\ \dot{p} \\ \dot{\lambda} \end{pmatrix}
	=
	\begin{pmatrix}
	\frac{\partial H}{\partial p}\\
	-\frac{\partial H}{\partial q} - (J g (q) )^{\top} \lambda\\
	g(q)
	\end{pmatrix}
	\end{equation}
	is considered. 
	Here, $ q : [0,T) \to \RR^n $ and $ p : [0,T) \to \RR^n $ are dependent variables, 
	$ H : \RR^n \times \RR^n \to \RR $ is the Hamiltonian, $ g : \RR^n \to \RR^h $ denotes the holonomic constraints ($ J g :\RR^d \to \RR^{h \times n}$ is the Jacobian matrix of $g$), and 
	$ \lambda: [0,T) \to \RR^h $ gives the corresponding Lagrange multipliers.
	Then, $ H $ is a conserved quantity: 
	\begin{align*}
	\frac{\rd}{\rd t} H(q(t),p(t)) 
	&= \left\langle \frac{\partial H}{\partial q} , \dot{q} \right\rangle + \left\langle \frac{\partial H}{\partial p} , \dot{p} \right\rangle \\
	&= \left\langle \frac{\partial H}{\partial q} , \frac{\partial H}{\partial p} \right\rangle + \left\langle \frac{\partial H}{\partial p} , -\frac{\partial H}{\partial q} - (J g (q) )^{\top} \lambda  \right\rangle \\
	&= - \left\langle \dot{q} ,  (J g (q) )^{\top} \lambda  \right\rangle 
	= - \left\langle (J g (q) ) \dot{q} , \lambda  \right\rangle \\
	&= - \left\langle \frac{\rd}{\rd t} g(q(t)) , \lambda \right\rangle = 0. 
	\end{align*}
	Moreover, 
	as the gradient $ \nabla H (q,p,\lambda) $ can be expressed by $ ( \partial H/ \partial q , \partial H / \partial p, 0 )^{\top} $, 
	the Hamiltonian $ H $ is proper. 
	
	On the other hand, the augmented Hamiltonian 
	\begin{equation}\label{eq_aug_hamiltonian}
	V(q,p,\lambda) = H(q,p) + \langle \lambda , g(q) \rangle
	\end{equation}
	is often used (see, e.g., \cite{G1999}). 
	Since $ V(q,p,\lambda) = H (q,p) $ holds for any $ (q,p,\lambda) \in \CM $, 
	the augmented Hamiltonian $ V $ is also a conserved quantity. 
	Furthermore, 
	because the gradient $ \nabla V(q,p,\lambda) $ can be expressed by $ ( \partial H/ \partial q + (J g (q))^{\top} \lambda , \partial H / \partial p , g(q) ) $, 
	the augmented Hamiltonian $ V $ is also a proper conserved quantity 
	(note that $ g(q) = 0 $ holds on~$ \CM $). 
\end{example}

\begin{example}\label{ex_ad_con}
	Furihata--Sato--Matsuo~\cite{FSM2016} considered a numerical method for evolutionary differential equations of the form
	\begin{equation}\label{eq_utx}
	u_{tx} = \frac{\delta \mathcal{H}}{\delta u},
	\end{equation}
	where subscripts denote partial derivatives, 
	$ u : [0,T) \times \Sb \to \RR $ is a dependent variable ($ \Sb = \RR / L \mathbb{Z} $), 
	$ t $ and $x $ are temporal and spatial independent variables, 
	and $ \delta \mathcal{H} / \delta u $ is the variational derivative of the functional $ \mathcal{H} $. 
	Note that $ \mathcal{H} $ is a conserved quantity: 
	\begin{equation*}
	\frac{\rd}{\rd t} \mathcal{H} (u(t)) =  \left\langle \frac{\delta \mathcal{H}}{\delta u } , u_t \right\rangle = \left\langle u_{tx} ,  u_t \right\rangle = 0, 
	\end{equation*}
	where $ \langle \cdot , \cdot \rangle $ is the standard $L^2$ inner product 
	(this is a slight abuse of notation, but it causes no confusion). 

	Though Furihata--Sato--Matsuo~\cite{FSM2016} dealt with full discretization, 
	we consider the corresponding spatial discretization. 
	By introducing the discrete symbol $ u_i (t) \approx u ( t, i \Delta x ) \ ( i = 1,\dots, I) $ (where $ \Delta x = 2\pi /I $ is the mesh size, the discrete periodic boundary condition  $ u_{i+I} = u_i $ is assumed, and $ u $ denotes the vector $ (u_1,\dots,u_I)^{\top} $), 
	the spatial discretization can be written in the form 
	\begin{equation}\label{eq_ad}
	D \dot{u} = M \nabla H (u),
	\end{equation}
	where $ D \in \RR^{I \times I}$ and $ M \in \RR^{ I \times I} $ are matrices representing the forward difference and average operators, i.e., 
	\begin{align*}
	D &= \frac{1}{\Delta x} \begin{pmatrix}
	-1 & 1 &  &  \\ 
	& -1 & \ddots &  \\ 
	&  & \ddots & 1 \\ 
	1 &  &  & -1
	\end{pmatrix},&
	M &= \frac{1}{2} \begin{pmatrix}
	1 & 1 &  &  \\ 
	& 1& \ddots &  \\ 
	&  & \ddots & 1 \\ 
	1 &  &  & 1
	\end{pmatrix},
	\end{align*}
	and $ H $ is an approximation of $ \mathcal{H} $. 
	
	As $ \mathbf{1}  \in (\Range (D) )^{\perp} $ and $ M^{\top} \mathbf{1} = \mathbf{1} $ hold, 
	\[ \langle \mathbf{1} , \nabla H (u) \rangle = \langle M^{\top} \mathbf{1} , \nabla H(u) \rangle = \langle \mathbf{1} , M \nabla H(u) \rangle = \langle \mathbf{1} , D \dot{u} \rangle = 0 \]
	must be satisfied for all solutions $ u $ of \eqref{eq_ad}. 
	This implies that $H$ is proper, i.e., $ \nabla H (u) \in \Car (D) $, because $ \Null (D) = \mathrm{span} \{ \mathbf{1} \} $. 
	
	Therefore, from~\eqref{eq_chain_proper}, it is easy to confirm that $ H $ is a conserved quantity:  
	\[ \frac{\rd}{\rd t} H (u) = \langle \nabla H (u) , D^{\dagger} M \nabla H (u) \rangle = 0. \]
	Here, the last equality comes from the skew-symmetry of $ D^{\dagger} M $, 
	which can be verified using the fact that all eigenvalues of $ D^{\dagger} M $ are purely imaginary and $ D^{\dagger} M $ is circulant. 
\end{example}


\section{Conservation/dissipation law and linear gradient DAEs}
\label{sec_lg}

To establish a unified framework for discrete gradient methods applied to DAEs, 
we should introduce the DAE counterpart of the linear gradient system~\eqref{eq_lg} for ODEs (recall Section~\ref{subsec_pre_dg}). 

To this end, 
we consider the linear gradient DAE
\begin{equation}\label{eq_lgdae}
A \dot{z} = S(z) \nabla V(z),
\end{equation}
which is a natural extension of linear gradient system~\eqref{eq_lg}. 
In this section, we explore the cases in which this has a conservation/dissipation law. 
Therefore, in order to use~\eqref{eq_chain_proper}, 
$V$ is assumed to be proper.

\subsection{Conservation law of DAEs}
\label{subsec_con_dae}

As illustrated in Example~\ref{ex_ad_con}, 
the linear gradient DAE~\eqref{eq_lgdae} is conservative if $ A^{\dagger} S $ is skew-symmetric. 

\begin{proposition}\label{prop_lg_con}
	Let $ V : \RR^d \to \RR $ be a proper function. 
	Then, $ V $ is a conserved quantity of~\eqref{eq_lgdae} if $ A^{\dagger} S (z) $ is skew-symmetric for any $ z \in \CM $. 
\end{proposition}

Moreover, as shown below, we can prove the converse of Proposition~\ref{prop_lg_con}, 
which gives an existence theorem of the linear gradient structure for conservative DAEs, 
i.e., a generalization of Proposition~\ref{prop_MQR_conservative} for DAEs. 

\begin{theorem}\label{thm_exist_con}
	Let $ f : \RR^d \to \RR^d $ be a $C^r$ map with $ r \ge 1 , \  d > 1 $, 
	and the $ C^{r+1} $ function $ V : \RR^d \to \RR $ be a proper conserved quantity. 
	Then, there exists a matrix function $ S $ such that $C^r$, $ f = S  \nabla V $ on the domain $ \{ z \in \RR^d \mid \nabla V(z) \neq 0  \} $, and $ A^{\dagger} S (z) $ is skew-symmetric for any $ z \in \CM $. 
	Moreover, $S$ can be chosen so as to be bounded near every non-degenerate critical point, that is, $S$ is locally bounded if $V$ is a Morse function.
\end{theorem}

\begin{proof}
	To prove the theorem, we confirm that all conditions are satisfied by $ S $ defined as follows: 
	\[ S(z) = \frac{ f(z) \left( \nabla V(z) \right)^{\top} - A \nabla V(z) \left( A^{\dagger} f(z) \right)^{\top} }{\| \nabla V(z) \|_2^2}. \]
	Note that $ S \nabla V (z) = f(z) $ holds, because $ V $ is a proper conserved quantity. 
	Moreover, the properness of $V$ implies that
	\[ A^{\dagger} S(z) =  \frac{ A^{\dagger} f(z) \left( \nabla V(z) \right)^{\top} - \nabla V(z) \left( A^{\dagger} f(z) \right)^{\top} }{\| \nabla V(z) \|_2^2}, \]
	and so $ A^{\dagger} S(z) $ is skew-symmetric. 
	
	Second, we consider the behavior of $S(z) $ in a neighborhood of a point $z$ for which $ \nabla V(z) = 0 $. 
	Though this is quite similar to the corresponding part of Proposition~\ref{prop_MQR_conservative}, 
	we give the full proof for the readers' convenience.  
	
	Under the Morse lemma, there is a coordinate chart about any non-degenerate critical point of $V$ in which
	\[ V(x) = V(0) + x^{\top} \nabla^2 V(0) x, \]
	where the Hessian $ \nabla^2 V(0) $ of $ V$ at $x = 0$ is non-degenerate.  
	For an arbitrary $x$ and sufficiently small $ \lambda $, 
	\[ \lambda x^{\top} \nabla^2 V(0) A^{\dagger} f (\lambda x) = (\nabla V(\lambda x) )^{\top} A^{\dagger} f(\lambda x) = 0 \]
	holds. This implies $ x^{\top} \nabla^2 V(0) A^{\dagger} f(0) = 0 $, so that the non-degeneracy of $\nabla^2 V(0)$ implies $ A^{\dagger} f(0) = 0 $. 
	
	As $ A^{\dagger} f $ vanishes at $ x = 0$, $ \| A^{\dagger} f(x) \|_2 / \| x \|_2 $ is locally bounded. 
	The non-degeneracy of $\nabla^2 V(0)$ implies that $ \| x \|_2 / \| \nabla V(x) \|_2 $ is also locally bounded. 
	Therefore, the inequality
	\[  \left| S_{ij}  \right| \le \frac{ \| f \|_2 \| \nabla V\|_2 }{\| \nabla V \|_2^2} + \frac{ \| A \nabla V \|_2 \| A^{\dagger} f \|_2 }{\| \nabla V\|_2^2}  \le 2 \| A\|_2 \frac{ \| A^{\dagger} f \|_2 }{ \| \nabla V\|_2 } \le 2 \| A\|_2 \frac{ \| A^{\dagger} f \|_2 }{\| x \|_2} \frac{ \| x \|_2 }{ \| \nabla V\|_2 }  \]
	implies that $ S $ is locally bounded. 
\end{proof}

Though Proposition~\ref{prop_lg_con} and Theorem~\ref{thm_exist_con} are natural extensions of the ODE case, 
they are very important in this context. 
As described in Examples~\ref{ex_hs_con} and \ref{ex_ad_con}, 
the proofs of conservation laws in existing results are quite different: 
the constraint is explicitly used in Example~\ref{ex_hs_con}, whereas the skew-symmetry of $D^{\dagger} M $ is essential in Example~\ref{ex_ad_con}. 
However, the above proposition and theorem imply that 
these conservation laws can be understood by the linear gradient DAE~\eqref{eq_lgdae} and skew-symmetry of $ A^{\dagger} S(z) $ (see Example~\ref{ex_hs_lg} below). 

\begin{example}\label{ex_hs_lg}
	For the constrained Hamiltonian system~\eqref{eq_hs}  
	and the Hamiltonian $H$, 
	the matrix $ S $ defined in Theorem~\ref{thm_exist_con} can be written in the form
	\begin{equation*}
	\frac{1}{h} \begin{pmatrix}
	H_p H_q^{\top} - H_q H_p^{\top} & H_q H_q^{\top} + H_p H_p^{\top} + H_q \lambda^{\top} G &  \\ 
	-H_q H_q^{\top} - H_p H_p^{\top} - G^{\top} \lambda H_q^{\top} & H_p H_q^{\top} - H_q H_p^{\top} + H_p \lambda^{\top} G - G^{\top} \lambda H_p^{\top} &  \\ 
	g(q) H_q^{\top} & g(q) H_p^{\top} & O_m
	\end{pmatrix},
	\end{equation*}
	where $ H_q = \partial H / \partial q $, $ H_p = \partial H/ \partial p $, $ G = J g(q) $, and $ h = \| H_q \|_2^2 + \| H_p \|_2^2 $. 
	Though it is easy to observe that the slightly simple matrix function 
	\[ S(q,p,\lambda ) = \frac{1}{h} \begin{pmatrix}
	O_n & h I + H_q \lambda^{\top} G &  \\ 
	-h I - G^{\top} \lambda H_q^{\top} & H_p \lambda^{\top} G - G^{\top} \lambda H_p^{\top} &  \\ 
	g(q) H_q^{\top} & g(q) H_p^{\top} & O_m
	\end{pmatrix} \]
	can be used, the linear gradient structure for the Hamiltonian $H$ is necessarily complicated 
	because the gradient of $H$ has no information about $ g $ and $ J g $. 
	
	However, for the augmented Hamiltonian $V$ defined by~\eqref{eq_aug_hamiltonian}, 
	the constrained Hamiltonian system~\eqref{eq_hs} has a simple linear gradient structure: 
	\begin{equation}\label{eq_hs_lg}
	\begin{pmatrix}
	I_n &  &  \\ 
	& I_n &  \\ 
	&  & O_h
	\end{pmatrix} 
	\begin{pmatrix} \dot{q} \\ \dot{p} \\ \dot{\lambda} \end{pmatrix}
	=
	\begin{pmatrix}
	O_n & I_n  &  \\ 
	- I_n & O_n &  \\ 
	&  & I_h
	\end{pmatrix}
	\nabla V(q,p,\lambda). 
	\end{equation}
	This linear gradient structure clearly satisfies the condition of Proposition~\ref{prop_lg_con}. 
\end{example}

In summary, 
for any conserved quantity $ \tilde{V} $ (see Remark~\ref{rem_hi_proper} for the handling of constraints), 
there exists a conserved quantity $V$ such that $ \tilde{V} = V $ holds on $ \CM $ (Proposition~\ref{prop_proper}), 
and there exists an appropriate linear gradient structure with respect to $V$ (Theorem~\ref{thm_exist_con}). 
Therefore, 
in principal, our linear gradient DAE~\eqref{eq_lgdae} with skew-symmetric $ A^{\dagger} S(z) $ can express all conservation laws for DAEs. 
Thus, we believe it is meaningful to consider the conservative discretization of such linear gradient DAEs (see Section~\ref{sec_dg}). 

\subsection{Dissipation law of DAEs}
\label{subsec_dis_dae}

In this section, we consider the dissipative case. 
Actually, we can deal with the dissipative case in a similar manner to the conservative case.  
However, we present the method in full for readers' convenience. 

First, let us introduce the notion of a dissipated quantity for DAEs of the form~\eqref{eq_dae}. 

\begin{definition}[Dissipated quantity]\label{def_proper_dl}
	Let $ V : \RR^d \to \RR $ be a $C^r$ map with $ r \ge 1$, $ d \ge 1 $. 
	Then, $V$ is said to be a \emph{dissipated quantity} if $ \frac{\rd}{\rd t} V(z(t)) \le 0 $ holds for any solution $z $ of DAE~\eqref{eq_dae}. 
\end{definition}

\begin{remark}\label{rem_lyapunov}
	There is a difference between the proper dissipated quantity used here and the existing definition of Lyapunov functions for DAEs in the literature. 
	Since our focus is on extending the concept of ODEs to DAEs, 
	we are not concerned with the additional properties of Lyapunov functions such as coercivity, boundedness, and strict dissipation. 
	
	Baji\'c~\cite{B1987} dealt with DAEs of the form $ A(t) \dot{z} = f(z) $ and considered their Lyapunov functions. 
	To overcome the difficulty that $ (\rd / \rd t) V(z(t)) $ cannot be expressed by means of $A$ and $f$ (recall (c) in Section~\ref{subsec_def_proper}), Baji\'c introduced a useful subclass of Lyapunov functions by adding the assumption that
	``the Lyapunov function $ V $ depends only on $t$ and $y$, where $y = \phi (t,z) $ is an auxiliary variable such that $ \dot{y} $ can be explicitly expressed by means of $ A $ and $f$.''
	In contrast, Liberzon--Trenn~\cite{LT2012} dealt with DAEs of the form~$ A(z) \dot{z} = f(z) $ under the assumption that 
	``there exists a continuous function $ \phi $ such that $ \langle \nabla V(z) , w \rangle =  \phi (z , A(z) w ) $ for any $ z \in \CM $ and $ w \in T_z \CM $.''
	
	The concept of properness is similar to the assumptions above. 
	However, as mentioned in the Introduction, the advantage of properness is that it does not lose generality.  
	This advantage truly relies on the autonomous nature of our target DAEs in comparison with~\cite{B1987}, 
	although properness can be extended to the case in  \cite{LT2012} under some assumption on $A(z)$. 
\end{remark}

Corresponding to Proposition~\ref{prop_lg_con}, 
the following proposition clearly holds for dissipative systems because of~\eqref{eq_chain_proper}. 

\begin{proposition}\label{prop_lg_dis}
	Let $ V : \RR^d \to \RR $ be a proper function. 
	Then, $ V $ is a dissipated quantity of~\eqref{eq_lgdae} if $ A^{\dagger} S (z) $ is negative semidefinite for any $ z \in \CM $. 
\end{proposition}

\begin{example}\label{ex_fric_dis}
	Uhler--Betch~\cite{UB2010} considered mechanical systems with linear friction written in the form
	\begin{equation}\label{eq_fric}
	\begin{cases}
	\dot{q} = v\\
	M \dot{v} = - \nabla U(q) - (J g(q) )^{\top} \lambda - F v \\
	g(q) = 0,
	\end{cases}
	\end{equation}
	where $ q \in \RR^n $ is the configuration vector, $ v \in \RR^n $ is the velocity, 
	$ M \in \RR^{n \times n} $ is the mass matrix ($M$ is assumed to be symmetric and positive definite), $ U : \RR^d \to \RR $ is a potential function, $ g : \RR^d \to \RR^h $ represents the holonomic constraints, and $ F v $ express the friction ($ F $ is a nonnegative diagonal matrix). 
	In this case, the energy function $ H (q,v) = \langle v , M v \rangle + U (q) $ is a dissipated quantity: 
	\begin{align*}
	\frac{\rd}{\rd t} H (q(t) , v(t))
	&= \left\langle \nabla U(q) , \dot{q} \right\rangle + \left\langle Mv , \dot{v} \right\rangle \\
	&= \left\langle \nabla U(q) , v \right\rangle + \left\langle v , - \nabla U(q) - (J g(q) )^{\top} \lambda - F v \right\rangle \\
	&= - \langle v, F v \rangle \le 0.
	\end{align*}
	Moreover, in a manner similar to Example~\ref{ex_hs_con}, we can confirm that $H$ is proper. 
	
	It is also easy to show that the augmented energy function 
 	$ V(q,v,\lambda) = H (q,v) + \langle \lambda , g (q) \rangle $
	is a proper dissipated quantity. 
\end{example}

As well as the conservative case, 
we can establish the converse of Proposition~\ref{prop_lg_dis}, 
which is a generalization of Proposition~\ref{prop_MQR_dissipative} for DAEs. 
However, in this case, we seek the matrix function $ S$ such that $ A^{\dagger} S (z) $ is negative semidefinite, whereas Proposition~\ref{prop_MQR_dissipative} ensures the existence of a negative definite matrix function. 
The singularity of $ A^{\dagger} $ means that this discrepancy cannot be solved. 
It should also be noted that 
we cannot ensure the boundedness of $S$ in general, 
whereas we can always construct a locally bounded $ S $ in the conservative case 
(see \cite[Proposition~2.9]{MQR1999}).

\begin{theorem}\label{thm_exist_dis}
	Let $ f : \RR^d \to \RR^d $ be a $C^r$ map with $ r \ge 1 , \  d\ge1 $, and the $ C^{r+1} $ function $ V : \RR^d \to \RR $ be a proper dissipated quantity.  
	Then, there exists a matrix function $ S $ such that $C^r$, $ f = S  \nabla V $ on the domain $ \{  z \in \RR^d \mid \langle A^{\dagger} f (z) , \nabla V(z) \rangle \neq 0  \} $, and $ A^{\dagger} S(z) $ is negative semidefinite for any $ z \in \CM$.  
\end{theorem}

\begin{proof}
	It is clear that the map $ S $ defined as 
	\begin{equation}
		S(z) = \frac{1}{\langle A^{\dagger} f (z) , \nabla V(z) \rangle} f(z) ( A^{\dagger} f(z) )^{\top} 
	\end{equation}
	satisfies all the conditions. 
\end{proof}

\begin{example}
	For mechanical systems with linear friction~\eqref{eq_fric}, 
	the linear gradient form with respect to the energy function $H$ is as complicated 
	as the case of the constrained Hamiltonian system (Example~\ref{ex_hs_lg}):
	\begin{equation}
	\begin{pmatrix} I_n & & \\ & M & \\ & & O_h \end{pmatrix}
	\begin{pmatrix} \dot{q} \\ \dot{v} \\ \dot{\lambda} \end{pmatrix}
	= - \frac{1}{ \langle v , F v \rangle }
	\begin{pmatrix} v v^{\top} & - v \Phi^{\top} M^{-1} & \\ - \Phi v^{\top} & \Phi \Phi^{\top} M^{-1} & \\ g(q) v^{\top} & g(q) \Phi^{\top} & O_h \end{pmatrix}
	\nabla H(q,v,\lambda),
	\end{equation}
	where $ \Phi :=  \nabla U (q) + (Jg (q))^{\top} \lambda + F v$. 

	However, for the augmented energy function $V$, 
	there is a relatively simple linear gradient form
	\begin{equation}
	\begin{pmatrix} I_n & & \\ & M & \\ & & O_h \end{pmatrix}
	\begin{pmatrix} \dot{q} \\ \dot{v} \\ \dot{\lambda} \end{pmatrix}
	=
	\begin{pmatrix} O_n & M^{-1} & \\ - I_n & - F M^{-1} & \\ & & I_h \end{pmatrix}
	\nabla V(q,v,\lambda).
	\end{equation}
	This linear gradient structure clearly satisfies the condition of Proposition~\ref{prop_lg_dis}. 
\end{example}

\section{Discrete gradient methods for linear gradient DAEs}
\label{sec_dg}

In the following, we focus on the conservative case (the dissipative case can be treated similarly). 

For linear gradient DAEs~\eqref{eq_lgdae}, 
a one-step method can be constructed as 
\begin{equation}\label{eq_lgdae_dg}
A \frac{\zd{m+1}{} -\zd{m}{}}{\Delta t } = \overline{S} \big( \zd{m+1}{}, \zd{m}{} \big) \overline{\nabla} V \big( \zd{m+1}{}, \zd{m}{} \big)
\end{equation}
using a discrete gradient $ \overline{\nabla} V $ and some consistent approximation $ S_{\rd} $ of $ S $. 
Here, as a discrete counterpart of Proposition~\ref{prop_lg_con}, the following proposition clearly holds. 

\begin{proposition}\label{prop_dg_con}
	Suppose that $ \overline{\nabla} V \big( \zd{m+1}{} , \zd{m}{} \big) \in \Car (A) $ holds and 
	$ A^{\dagger} \overline{S} \big( \zd{m+1}{} , \zd{m}{} \big) $ is skew-symmetric 
	for a numerical solution $ \zd{m+1}{} $ of the discrete gradient method~\eqref{eq_lgdae_dg}. 
	Then, $ V \big( \zd{m+1}{} \big) = V \big( \zd{m}{} \big) $ holds.
\end{proposition}

Though the definition of discrete gradient methods and Proposition~\ref{prop_dg_con} are quite natural extensions of the ODE case (see Section~\ref{subsec_pre_dg}), 
the assumptions in Proposition~\ref{prop_dg_con} are troublesome: 
\begin{enumerate}
	\item $ \overline{\nabla} V \big( \zd{m+1} , \zd{m}{} \big) \in \Car (A) $ cannot be ensured in general:\\
	Even when $ V $ is a proper function, its discrete gradient need not belong to $ \Car (A)$. 
	For example, the nonlinear nature of $ \CM $ means that $ (1- \xi) z + \xi z' \in \CM $ cannot be guaranteed, even when $ z , z' \in \CM $ and $ \xi \in (0,1) $. Hence, the average vector field~\eqref{eq_avf} cannot be guaranteed to belong to $ \Car (A) $ (we consider how to overcome this issue in the next section). 
	\item An appropriate discretization of $S$ can become nontrivial:\\
	In the continuous case, the skew-symmetry is only ensured on $ \CM $. 
	Thus, if the numerical solution $ \zd{m}{} $ satisfies $ \zd{m}{} \in \CM $, $ \overline{S} \big( \zd{m+1}{},\zd{m}{} \big) = \big( S \big( \zd{m+1}{} \big) + S \big( \zd{m}{} \big) \big)/2 $ is an appropriate approximation (if symmetry is not required, one can choose a simpler discretization $ \overline{S} \big( \zd{m+1}{},\zd{m}{} \big) = S \big( \zd{m}{} \big) $). 
	However, when we use a numerical scheme that does not satisfy $ \zd{m}{} \in \CM $, 
	the situation becomes far more challenging. 
	Fortunately, however, a number of examples have constant $ S $. 
	Thus, we leave this issue for future work. 
\end{enumerate}

We can overcome these difficulties in existing cases as shown in the following examples. 

\begin{example}\label{ex_hs_dg}
	For the linear gradient form~\eqref{eq_hs_lg}, we consider the discrete gradient method 
	\begin{equation}\label{eq_hs_dg}
	\frac{1}{\Delta t}
	\begin{pmatrix}
	I_n &  &  \\ 
	& I_n &  \\ 
	&  & O_h
	\end{pmatrix} 
	\left(
	\begin{pmatrix} q^{(m+1)} \\ p^{(m+1)} \\ \lambda^{(m+1)} \end{pmatrix}
	-
	\begin{pmatrix} q^{(m)} \\ p^{(m)} \\ \lambda^{(m)} \end{pmatrix} \right)
	=
	\begin{pmatrix}
	O_n & I_n  &  \\ 
	- I_n & O_n &  \\ 
	&  & I_h
	\end{pmatrix}
	\overline{\nabla} V. 
	\end{equation}
	Here, if we employ the discrete gradient in the form
	\begin{equation}
	\overline{\nabla} V = \begin{pmatrix} \overline{\nabla}_q H + (\overline{J} g )^{\top} \frac{\lambda^{(m+1)} + \lambda^{(m)}}{2} \\ \overline{\nabla}_p H \\ \frac{g(q^{(m+1)}) + g(q^{(m)}) }{2} \end{pmatrix},
	\end{equation}
	the discrete gradient method~\eqref{eq_hs_dg} coincides with the scheme employed by Gonzalez~\cite{G1999} under the assumption $ g(q^{(0)}) = 0 $ 
	(we omit several obvious arguments for brevity). 
	$ ( ( \overline{\nabla}_q H)^{\top} , ( \overline{\nabla}_p H )^{\top} )^{\top}$ is a discrete gradient of $ H $, 
	and the discrete Jacobian $ \overline{J} g $ is defined in a similar manner to the discrete gradient. 
	In this case, \cite[Proposition~3.1]{G1999} ensures that $ H $ is preserved and $ g \big( q^{(m)} \big) = 0 $ holds for all $ m$. 
	
	We can also prove the discrete conservation law using Proposition~\ref{prop_dg_con}. 
	First, note that $ g \big( q^{(m+1) } \big) +  g \big( q^{(m) } \big) = 0 $ implies $ g \big(q^{(m+1)} \big) = 0 $ (under the assumption $ g \big( q^{(m) } \big) = 0 $). 
	Since this ensures $ \overline{\nabla} V \in \Car (A) $, all conditions of Proposition~\ref{prop_dg_con} are fulfilled, and so $ V \big(q^{(m+1)}, p^{(m+1)}, \lambda^{(m+1)}\big) = V \big(q^{(m)}, p^{(m)}, \lambda^{(m)}\big) $ holds for all $m$. 
	Moreover, the definition of $V$ and $ g \big( q^{(m) } \big) = 0 $ implies the conservation of $H$. 
	
	Note that $ (q^{(m)} , p^{(m)}, \lambda^{(m)} ) \in \CM $ does not hold in general. 
	It is difficult to overcome this issue, because the constrained Hamiltonian system has a higher index (at least 3). 
\end{example}

\begin{example}\label{ex_ad_dg}
	For the spatial discretization~\eqref{eq_ad} of the variational PDE~\eqref{eq_utx}, 
	we consider the discrete gradient method
	\begin{equation}\label{eq_ad_dg}
	D \frac{\ud{m+1}{} - \ud{m}{}}{\Delta t } = M \overline{\nabla} H \big( \ud{m+1}{} , \ud{m}{} \big),
	\end{equation}
	which coincides with the full discretization employed in~\cite{FSM2016}. 
	
	In this case, similar to the continuous (semi-discrete) case, 
	we can show that $ \overline{\nabla} H \big( \ud{m+1}{} , \ud{m}{} \big) \in \Car (D) $: 
	\[ \langle \mathbf{1} , \overline{\nabla} H  \rangle = \langle M^{\top} \mathbf{1} , \overline{\nabla} H \rangle = \langle \mathbf{1} , M \overline{\nabla} H \rangle = \left\langle \mathbf{1} , D  \frac{\ud{m+1}{} - \ud{m}{}}{\Delta t } \right\rangle = 0. \]
	Thus, Proposition~\ref{prop_dg_con} reveals that $ H $ is conserved by the numerical method~\eqref{eq_ad_dg}.
	
	Note that, in general, the implicit constraint is not satisfied, 
	i.e., $ \ud{m}{} \notin \CM $. 
\end{example}

Though the above examples have the desired conservation law, 
the general framework seems to be quite challenging because of the difficulties identified; 
the numerical solution does not belong to $ \CM $.  
Therefore, in the next section, we focus on index-1 cases, 
and construct a general conservative method using the discrete gradient. 
There, in contrast to the above examples, 
we force the numerical solution to belong to $ \CM $ (this can be done thanks to the index-1 assumption). 
This implies the latter difficulty no longer occurs, 
and we can focus on the former issue.

\section{Discrete gradient methods for index-1 DAEs}
\label{sec_dg_index1}

In this section, to overcome the difficulties identified in the previous section, 
we focus on index-1 cases. 
To establish the general result, 
we first consider a new discrete gradient that is compatible with proper conserved quantities. 
We then establish a new discrete gradient method. 

\subsection{A new discrete gradient}

It is important to observe that 
$ \nabla V(\zd{m}{}) , \nabla V(\zd{m+1}{} ) \in \Car (A) $ holds when we assume $ \zd{m}{} , \zd{m+1}{} \in \CM $ for a proper function $V$. 
Therefore,  if we can construct the discrete gradient defined by some interior division of the vectors $ \nabla V(\zd{m}{}) $ and $ \nabla V(\zd{m+1}{}) $, 
the resulting discrete gradient also belongs to $ \Car (A) $. 
However, as far as the present author knows, such a discrete gradient has not been discussed in the literature. 
In this section, we show such a discrete gradient can be constructed for some cases. 

For $V: \RR^d \to \RR $, we define
\begin{equation}
\dgp V (z , z' ) = \begin{cases} \theta ( z , z' ) \nabla V(z ) + \theta (z' , z) \nabla V(z') & (z \neq z' ), \\ \nabla V(z) & (z = z' ) \end{cases} \label{eq_dgp} 
\end{equation}
where the coefficient $ \theta : \RR^d \times \RR^d \to \RR $ is defined by
\begin{equation}\label{eq_dg_coeff}
\theta (z, z' ) = \frac{ V(z) - V(z') - \langle \nabla V(z') , z - z' \rangle }{\langle \nabla V(z) - \nabla V(z') , z -z' \rangle }. 
\end{equation}

First, we confirm the discrete chain rule (recall Definition~\ref{def_dg}): 
	\begin{align*}
	\langle \dgp V(z,z') , z - z' \rangle 
	&= \theta (z , z') \langle \nabla V(z ) , z -z' \rangle  + \theta ( z' , z) \langle \nabla V(z') , z - z' \rangle \\
	&= \frac{V(z) - V(z') }{\langle \nabla V(z) - \nabla V(z') , z - z' \rangle } \left( \langle \nabla V(z ) , z - z' \rangle - \langle \nabla V(z') , z - z' \rangle \right) \\
	& \quad - \frac{ \langle \nabla V(z') , z - z' \rangle 
		\langle \nabla V(z ) , z - z' \rangle +  \langle \nabla V( z ) , z' - z \rangle \langle \nabla V(z') , z - z' \rangle }{\langle \nabla V(z) - \nabla V(z') , z - z' \rangle } \\
	&= V(z) - V(z'). 
	\end{align*}
The second property in Definition~\ref{def_dg} is satisfied by definition, 
and $ \dgp V(z,z') $ is clearly symmetric. 

However, in general, as the denominator of $ \theta (z,z') $ can be zero even when $ z \neq z' $,  
the continuity of $ \dgp V(z,z') $ depends on $V$. 
For example, $ \dgp V $ is a continuous function for quadratic or strictly convex function $V$: 
\begin{itemize}
	\item When $ V $ is quadratic, i.e., $ V(z) = (1/2) z^{\top} X z  $ for a symmetric matrix $ X $:\\
	Because
	\begin{align*}
	\theta (z, z') &= \frac{ \frac{1}{2} ( \langle z , X z \rangle - \langle z' , X z' \rangle ) - \langle X z' , z - z' \rangle  }{\langle X ( z - z' ) , z - z' \rangle }
	= \frac{ \langle  X z , z \rangle - \langle X z' , 2 z - z' \rangle }{2 \langle X ( z - z' ) , z - z' \rangle } \\
	&= \frac{ \langle  X (z - z') ,z - z'  \rangle }{2 \langle X ( z - z' ) , z - z' \rangle } = \frac{1}{2},
	\end{align*}
	$ \dgp V(z,z') = B ( z + z')/2 $ holds, which coincides with the average vector field. 
	\item When $ V $ is strictly convex:\\
	The function $V$ is said to be strictly convex if 
	$ V( (1-\xi) z + \xi z' ) < ( 1 - \xi ) V(z) + \xi V(z') $
	holds for any $ z \neq z' $ and $ \xi \in (0,1) $. 
	In this case, since 
	$ V(z) > V(z') + \langle \nabla V(z') , z - z' \rangle $
	holds for any $ z \neq z' $, the denominator of $ \theta (z,z') $ can be estimated by 
	\begin{align*}
	\langle \nabla V(z) - \nabla V(z') , z - z' \rangle 
	&= \langle \nabla V(z) , z - z' \rangle + \langle \nabla V(z') , z' - z \rangle \\
	&> \left( V(z) - V(z') \right) + \left( V(z') - V(z) \right) = 0. 
	\end{align*}
	Therefore, it is positive and accordingly continuous on the domain $ \{ (z,z') \mid z \neq z' \} $. 
	The continuity on the whole domain can also be proved (see appendix). 
\end{itemize}

As the discrete gradient $ \dgp V $ is defined by the interior division, 
the following lemma clearly holds. 

\begin{lemma}\label{lem_dg_proper}
	Let $ z , z' \in \RR^d $ be elements of $ \CM $, and $V$ be a proper function. 
	Then, $ \dgp V(z,z') \in \Car (A) $ holds. 
\end{lemma}

\subsection{Discrete gradient method}

To ensure $ \zd{m}{} \in \CM $, we consider a reformulation of the linear gradient DAE~\eqref{eq_lgdae} in the following form: 
\begin{equation}\label{eq_dae_red}
\begin{cases}
A \dot{z} = S (z) \nabla V(z) + \sum_{i=1}^{\ell} c_i b_i, \\
G(z) = 0,
\end{cases}
\end{equation} 
where $ c_i \in \RR $ is a redundant variable (which turns out to be identically zero) for each $  i = 1, \dots , \ell$, 
and $ G(z) = B^{\top} S(z) \nabla V(z) $ is an implicit constraint. 

We consider the following one-step method using the discrete gradient~$ \dgp V $ introduced in the previous section: 
\begin{equation}\label{eq_dg_red}
\begin{cases}
A \frac{\zd{m+1}{} - \zd{m}{} }{\Delta t } = \overline{S} \big( \zd{m+1}{} , \zd{m}{} \big) \dgp V \big( \zd{m+1}{} , \zd{m}{} \big) + \sum_{i=1}^{\ell} c_i^{(m+1)} b_i, \\
G \big( \zd{m+1}{} \big) = 0, 
\end{cases}
\end{equation}
where $ S(\zd{m+1}{} , \zd{m}{} ) = (S( \zd{m+1}{} ) + S(\zd{m}{} ) )/2 $. 
Here, we assume that the scheme~\eqref{eq_dg_red} has a unique solution for sufficiently small $ \Delta t $. 
By definition, the following lemma obviously holds. 

\begin{lemma}\label{lem_dg_ic}
	Let $ \zd{m}{} $ be the numerical solution of~\eqref{eq_dg_red} ($m = 0, 1,\dots $). 
	Then, $ \zd{m}{} \in \CM $ holds for any $m$. 
\end{lemma}

\begin{proof}
	As we assume the DAE has a uniform index-1, $ \CM = \{ z \in \RR^d \mid G(z) = 0 \} $ holds. 
	This fact implies the lemma. 
\end{proof}

Using Lemmas~\ref{lem_dg_proper} and \ref{lem_dg_ic}, 
we obtain the desired discrete conservation law as follows. 

\begin{theorem}
	Let $ \zd{m}{} $ be the numerical solution of~\eqref{eq_dg_red} ($m = 0, 1,\dots $). 
	Then, $ V(\zd{m+1}{}) = V(\zd{m}{} ) $ holds for any $ m $. 
\end{theorem}

\begin{proof}
	For the numerical solution $ \zd{m}{} $, 
	there exists some $ e^{(m+1)} \in \Null (A) $ such that 
	\[ \frac{\zd{m+1}{} - \zd{m}{} }{\Delta t } = A^{\dagger} \overline{S} \big( \zd{m+1}{} , \zd{m}{} \big) \dgp V \big( \zd{m+1}{} , \zd{m}{} \big) + e^{(m+1)} \]
	(note that $ A^{\dagger} b_i = 0 $ holds because $ \Null (A^{\dagger} ) = ( \Range (A) )^{\perp} $). 
	Therefore, we see that
	\begin{align*}
	\frac{V(\zd{m+1}{} ) - V(\zd{m}{} ) }{\Delta t}
	&= \left\langle \dgp V(\zd{m+1}{} ,\zd{m}{} ) , \frac{\zd{m+1}{}  - \zd{m}{}}{\Delta t} \right\rangle \\
	&= \left\langle \dgp V(\zd{m+1}{} ,\zd{m}{} ) , A^{\dagger} \overline{S} \big( \zd{m+1}{} , \zd{m}{} \big) \dgp V \big( \zd{m+1}{} , \zd{m}{} \big) + e^{(m+1)} \right\rangle \\
	&= \left\langle \dgp V(\zd{m+1}{} ,\zd{m}{} ) , A^{\dagger} \overline{S} \big( \zd{m+1}{} , \zd{m}{} \big) \dgp V \big( \zd{m+1}{} , \zd{m}{} \big) \right\rangle \\
	&= 0,
\end{align*}
	where the last equality comes from the skew-symmetry of $ A^{\dagger} \overline{S} (\zd{m+1}{} , \zd{m}{} ) $ (recall that the sum of skew-symmetric matrices is again skew-symmetric), 
	and the third equality comes from $ e^{(m+1)} \in \Null (A) $ and $ \dgp V(\zd{m+1}{} ,\zd{m}{} ) \in \Car (A) $. 
	Note that $ \dgp V(\zd{m+1}{} ,\zd{m}{} ) \in \Car (A) $ is ensured by $ \zd{m+1}{} , \zd{m}{} \in \CM $ and Lemma~\ref{lem_dg_proper}. 
\end{proof}

\section{Numerical example}
\label{sec_num}


In this section, we examine the case of the sinh-Gordon equation 
\begin{equation}\label{eq_shG}
u_{tx} = \sinh u,
\end{equation}
which is a special case of~\eqref{eq_utx} with $ \mathcal{H} (u) = \int_{\Sb} \cosh u \, \rd x $. 
The solution of the sinh-Gordon equation must satisfy the implicit constraint $ \mathcal{F} (u) = \int_{\Sb} \sinh u \, \rd x = 0 $ (see \cite{SM2017arx} for details on such implicit constraints).

In this case, the discrete energy can be defined as $ H (u) = \sum_{i=1}^I \cosh (u_i) $ 
(though $ \overline{\mathcal{H}} (u) := H(u) \Delta x $ is an appropriate discretization of $\mathcal{H}$, 
$H$ is employed to ensure that $ \nabla H $ is an approximation of $ \delta \mathcal{H} / \delta u $; see \cite{CGMMOOQ2012} for this point). 
Then, because $ H $ is strictly convex, the discrete gradient $ \dgp H $ must be a continuous function. 
Using this fact, the discrete gradient method 
\begin{equation}\label{eq_shg_dvdm_prop}
\begin{cases}
D \frac{\ud{m+1}{} - \ud{m}{} }{\Delta t } = M \dgp H \big( \ud{m+1}{} , \ud{m}{} \big) + c^{(m+1)} \mathbf{1}, \\
F(\ud{m+1}{} ) := \sum_{i=1}^I \sinh ( \ud{m+1}{i} ) = 0, 
\end{cases}
\end{equation}
can be defined according to~\eqref{eq_dg_red}. 
In this case, multiplying by $ \mathbf{1}^{\top} $, the first equation implies that
\[ 0 = \mathbf{1}^{\top} D \frac{\ud{m+1}{} - \ud{m}{} }{\Delta t } = \mathbf{1}^{\top} \dgp H \big( \ud{m+1}{} , \ud{m}{} \big) + I c^{(m+1)} = I c^{(m+1)}, \]
i.e., $ c^{(m+1)} = 0 $ holds, 
whereas the last equality comes from $ F ( \ud{m}{} ) = F(\ud{m+1}{} ) = 0 $. 
Therefore, the simple discrete gradient method 
\begin{equation}\label{eq_shg_dvdm_r}
D \frac{\ud{m+1}{} - \ud{m}{} }{\Delta t } = M \dgp H \big( \ud{m+1}{} , \ud{m}{} \big) 
\end{equation}
turns out to be mathematically equivalent to~\eqref{eq_shg_dvdm_prop} 
as long as we assume $ \zd{0}{} \in \CM $, i.e., $ F ( \ud{0}{} ) = 0 $. 
Therefore, this scheme (using $ \dgp H $ for the usual discrete gradient scheme~\eqref{eq_ad_dg}) actually conserves $F$ and $H$. 
However, of course, if we employ the average vector field as a discrete gradient, i.e., we consider
\begin{equation}\label{eq_shg_dvdm_avf}
D \frac{\ud{m+1}{} - \ud{m}{} }{\Delta t } = M \overline{\nabla}_{\mathrm{AVF}} H \big( \ud{m+1}{} , \ud{m}{} \big) , 
\end{equation}
the scheme only conserves $ H $. 

To compare these approaches numerically, we use the periodic traveling wave solution discovered by Li--Yin~\cite{LY2017_1}. 
We omit the concrete form of the solution, but its shape is shown in Fig.~\ref{fig_ShG_initial}. 
To solve the nonlinear equations, we simply employ the `fsolve' function in MATLAB. 
We set the parameters to $ I = 128 $, $ \Delta t = 0.1 $, $ T = 10 $, and $ L \approx 5.91 $ (the period $L$ is numerically computed because of the cumbersome definition of the exact solution). 
As shown in Fig.~\ref{fig_ShG_ns}, the nonlinear equation~\eqref{eq_shg_dvdm_prop} is solved very accurately (recall that $ c^{(m)} = 0 $ is the exact solution), 
and other nonlinear equations are also solved well.

\begin{figure}[htp]
\centering
\begin{minipage}{0.48\textwidth}
	\centering
	\pgfplotsset{width=7cm,compat=newest}
	\begin{tikzpicture}
	\begin{axis}[xlabel = {$x$}, enlarge x limits = false, ylabel= {$u$}]
	\addplot[black] table[x index = 0, y index = 2] {data/ShGDVDM0.dat};
	\end{axis}
	\end{tikzpicture}
	\caption{Initial condition; traveling wave. }
	\label{fig_ShG_initial}
\end{minipage}
\hspace{5pt}
\begin{minipage}{0.48\textwidth}
	\centering
	\begin{tikzpicture}
	\begin{semilogyaxis}[width=7cm,compat = newest,xlabel={$t$},enlarge x limits=false, 
	ylabel ={$ c^{(m)} $}]
	\addplot[smooth,red] table {data/ShGDVDMns.dat};
	\end{semilogyaxis}
	\end{tikzpicture}
	\caption{Evolution of $ c^{(m)} $ in scheme~\eqref{eq_shg_dvdm_prop}.}
	\label{fig_ShG_ns}
\end{minipage}
\vspace{5pt}

\begin{minipage}{0.48\textwidth}
	\centering
	\begin{tikzpicture}
	\begin{semilogyaxis}[width=7cm,compat = newest,xlabel={$t$},enlarge x limits=false, 
	ylabel ={error}]
	\addplot[smooth,red] table {data/ShGDVDMenerror.dat};
	\addplot[smooth,blue] table {data/ShGDVDMHenerror.dat};
	\addplot[smooth,green] table {data/ShGDVDMRenerror.dat};
	\end{semilogyaxis}
	\end{tikzpicture}
	\caption{Evolution of errors in the energy $ \overline{\mathcal{H}} = H \Delta x $ for each numerical solution. Solid lines in red, green, and blue correspond to schemes~\eqref{eq_shg_dvdm_prop}, \eqref{eq_shg_dvdm_r}, and \eqref{eq_shg_dvdm_avf}, respectively}
	\label{fig_ShG_en}
\end{minipage}
\hspace{5pt}
\begin{minipage}{0.48\textwidth}
	\centering
	\begin{tikzpicture}
	\begin{semilogyaxis}[width=7cm,compat = newest,xlabel={$t$},enlarge x limits=false, 
	ylabel ={error}]
	\addplot[smooth,red] table {data/ShGDVDMic.dat};
	\addplot[smooth,blue] table {data/ShGDVDMHic.dat};
	\addplot[smooth,green] table {data/ShGDVDMRic.dat};
	\end{semilogyaxis}
	\end{tikzpicture}
	\caption{Evolution of errors in the implicit constraint $ \overline{\mathcal{F}} = F \Delta x $ for each numerical solution. Solid lines in red, green, and blue correspond to schemes~\eqref{eq_shg_dvdm_prop}, \eqref{eq_shg_dvdm_r}, and \eqref{eq_shg_dvdm_avf}, respectively. }
	\label{fig_ShG_ic}
\end{minipage}
\end{figure}

Schemes~\eqref{eq_shg_dvdm_prop}, \eqref{eq_shg_dvdm_r}, and~\eqref{eq_shg_dvdm_avf} all reproduce the traveling wave perfectly (omitted), 
and conserve the energy~$ \mathcal{H} $ very well (see Fig.~\ref{fig_ShG_en}). 
However, the average vector field~\eqref{eq_shg_dvdm_avf} is slightly worse in view of the implicit constraint 
(Fig.~\ref{fig_ShG_ic}). 
This result agrees very well with the theory described above. 

Note that, in view of computational efficiency, 
the average vector field~\eqref{eq_shg_dvdm_avf} is the best scheme because it is free of nonlocal computations (see Table~\ref{tab_ct}).
Recall that the calculation of $ \dgp H (\ud{m+1}{} , \ud{m}{} )  $ requires nonlocal computations (see definition of $\theta $ in~\eqref{eq_dg_coeff}). 
The actual computation times are summarized in Table~\ref{tab_ct}. 
Though not especially meaningful in this quite simple implementation, 
the results imply that scheme~\eqref{eq_dg_red} is computationally expensive. 

Though schemes~\eqref{eq_shg_dvdm_prop} and~\eqref{eq_shg_dvdm_r} are mathematically equivalent, 
there are some visible differences between them in Figures~\ref{fig_ShG_en} and \ref{fig_ShG_ic}. 
The cause of this phenomenon is the very strong assumption that $ F(\zd{0}{} ) = 0 $ for the equivalence. 
In fact, this assumption cannot be ensured exactly in actual numerical computations. 
Moreover, even when it is ensured, $ F ( \zd{m}{} ) = 0 $ does not exactly hold because of the inevitable round-off error for $ m \ge 1 $. 

\begin{table}[htp]
	\centering
	\caption{Computation time for each numerical scheme. }
	\label{tab_ct}
	\begin{tabular}{c||c|c|c}
		\hline
		Scheme & \eqref{eq_shg_dvdm_prop} & \eqref{eq_shg_dvdm_r} & \eqref{eq_shg_dvdm_avf} \\ \hline 
		Time (s) & 7.6292 & 6.1884 & 4.3392 \\ \hline
	\end{tabular}
\end{table}

\section{Concluding remarks}
\label{sec_cr}

In this paper, we have presented the first steps toward a unified framework for the discrete gradient method applied to DAEs. 

As the first building block of such a framework, 
we showed that the linear gradient DAE~\eqref{eq_lgdae} is an appropriate class for considering discrete gradient methods (Section~\ref{sec_lg}). 
To overcome several difficulties, we introduced the concept of proper functions. 
Though unusual at first sight, 
we believe the concept of properness is indispensable in considering conservation/dissipation laws in DAEs. 

We then discussed the difficulty of constructing conservative methods for general conservative DAEs 
(Section~\ref{sec_dg}). 
Finally, we derived a partial answer for index-1 cases in Section~\ref{sec_dg_index1}, 
and confirmed this result numerically in Section~\ref{sec_num}. 

However, several issues remain. 
First, higher index cases seem to be quite difficult to deal with (Section~\ref{sec_dg}), 
but should be investigated. 
Second, as there are numerous equivalent functions (i.e., whose values are the same on $ \CM $), 
determining the best one is an interesting issue 
(recall that the augmented Hamiltonian yields the better linear gradient structure (Example~\ref{ex_hs_lg})). 
Finally, the author believes that 
the discrete gradient scheme defined in Section~\ref{sec_dg_index1}
should be investigated in more detail, 
because it simultaneously preserves the conserved quantity and the implicit constraint. 
Since the preservation of multiple conserved quantities has presented a difficult task previous studies, 
the author hopes that this contribution will provide some progress in this direction.

\section*{Acknowledgment}
	The author is grateful to Kensuke Aishima and Takayasu Matsuo for valuable comments.


\section*{Appendix}
\appendix

\section{Continuity of $ \dgp V $ for strictly convex $ V$}

To demonstrate the continuity of $ \dgp V $ around $ z = z' $, 
we fix the point $ z^{\ast} \in \RR^d $ and set $ z = z^{\ast} + \epsilon $ and $ z' = z^{\ast} $. 
Then, we consider the Taylor expansion of 
the denominator and numerator of $ \theta (z^{\ast} + \epsilon , z^{\ast} ) $ with respect to $ \epsilon $. 

The numerator and denominator can be expanded as
\begin{align*}
 V(z^{\ast} + \epsilon) - V(z^{\ast} ) - \langle \nabla V(z^{\ast}) , \epsilon \rangle 
 &= \frac{1}{2} \sum_{i,j} \frac{\partial^2 V}{ \partial z_i \partial z_j } \epsilon_i \epsilon_j + \frac{1}{6} \sum_{i,j,k} \frac{\partial^3 V }{ \partial z_i \partial z_j \partial z_k} \epsilon_i \epsilon_j \epsilon_k + \cdots, \\
 \langle \nabla V(z^{\ast} + \epsilon) - \nabla V(z^{\ast}) , \epsilon \rangle 
 &= \sum_{i,j} \frac{\partial^2 V}{ \partial z_i \partial z_j } \epsilon_i \epsilon_j + \frac{1}{2} \sum_{i,j,k} \frac{\partial^3 V }{ \partial z_i \partial z_j \partial z_k} \epsilon_i \epsilon_j \epsilon_k + \cdots. 
 \end{align*}
Because each term of the Taylor expansions coincides except for a constant, 
 $ \theta ( z^{\ast} + \epsilon , z^{\ast} ) $ is bounded above in a neighborhood of $ \epsilon = 0 $ (note that the denominator is positive when $ \epsilon \neq 0 $). 
Moreover, the definition of $ \dgp V $ can be rewritten in the form 
 \[ \dgp V(z,z') = \nabla V(z') + \theta ( z, z') \left( \nabla V(z) - \nabla V(z') \right). \]
 Therefore, in this case, $ \dgp V $ is continuous everywhere. 

\end{document}